\documentclass{article}
\def\Bbb{\mathbb}
\def\cal{\mathcal}
 \catcode`@=11
\usepackage{amsthm,amssymb,amsfonts,mathrsfs}
\usepackage{amsmath}
\catcode`@=11 \@addtoreset{equation}{section}

\catcode`@=12
\newtheorem{Theorem}{Theorem}[section]
\newtheorem{Proposition}{Proposition}[section]
\newtheorem{Lemma}{Lemma}[section]
\newtheorem{Corollary}{Corollary}[section]
\theoremstyle{definition}

\newtheorem{Remark}{Remark}

\newtheorem{Assumptions}{Hypothesis}[section]

\def\cH{\mathcal H}

\def\ds{\displaystyle}

\def\fd{\mathfrak{d}}
\def\fg{\mathfrak{g}}
\def\fh{\mathfrak{h}}

\def\ve{\varepsilon}
\def\V1{\mathcal{V}_1}
\def\V2{\mathcal{V}_2}

\def\io{\int_\omega}

\def\ve{\varepsilon}

\title {Carleman estimates and null controllability for a degenerate population model}
\author{{\sc Genni Fragnelli}\\
Dipartimento di Matematica\\ Universit\`{a} di Bari "Aldo Moro"\\
Via
E. Orabona 4\\ 70125 Bari - Italy\\ email: genni.fragnelli@uniba.it}

\date{}
\begin{document}

\maketitle

\vspace{0.3cm}

\centerline{ {\it  }}

\begin{abstract}
We deal with a degenerate model describing the dynamics of a population depending on time, on age and on space. We assume that the degeneracy can  occur at the boundary or in the interior  of the space domain and we focus on null controllability problem. To this aim, we prove first Carleman estimates for the associated adjoint problem, then, via cut off functions, we prove the existence of a null control function localized in the interior of the space domain.
\end{abstract}



\section{Introduction}

We consider the following linear population model describing the dynamics of a single species:
\begin{equation} \label{1}
\begin{cases}
\frac{\partial y}{\partial t}+\frac{\partial y}{\partial a}
-k(x)y_{xx}+\mu(t, a, x)y =f(t,a,x)\chi_{\omega} & \quad \text{in } Q,\\
  y(t, a, 1)=y(t, a, 0)=0 & \quad \text{on }Q_{T,A},\\
 y(0, a, x)=y_0(a, x) &\quad \text{in }Q_{A,1},\\
 y(t, 0, x)=\int_0^A \beta (a, x)y (t, a, x) da  &\quad  \text{in } Q_{T,1},
\end{cases}
\end{equation}
in the domain $Q:=(0,T)\times(0,A)\times(0,1)$. Moreover, $Q_{T,A} := (0,T)\times (0,A)$, $Q_{A,1}:=(0,A)\times(0,1)$ and
$Q_{T,1}:=(0,T)\times(0,1)$. Here $y(t,a,x)$  is the distribution of certain individuals of age $a\in (0,A)$ at time $t\in(0,T)$ and location $x \in (0,1)$, while $ \chi_\omega$ is the characteristic function of  $\omega\subset(0,1)$, which is the region where the control $f$ acts; $A$ is the maximal age of life, and $\beta$ and $\mu$ are the natural fertility and the death rate, respectively. Thus, the formula $\int_0^A \beta y da$ denotes the distribution of newborn individuals at time $t$ and location $x$.  The function $k$ is the dispersion coefficient and we assume that it depends on the space variable $x$ and can degenerate at the boundary  or in the interior of the state space.

In the last centuries, population models have been widely investigated by many authors from many points of view (see, for example, \cite{diekmann}, \cite{diekmann1}, \cite{rhandi}, \cite{fim}, \cite{fmvM3}, \cite{ft}, \cite{langlaisJDE}, \cite{langlaisSIAM}, \cite{rhandi1}, \cite{tello}, \cite{turanova}, \cite{webb}). In particular, one of the most studied problem has been the controllability of the system. Indeed, $y$ can represent the distribution of a {\it demaging insect population} or of a {\it pest population} (see, for example, \cite{he}), thus it is important to control it. For example in \cite{he}, where \eqref{1} models an insect growth, the control corresponds to a removal of individuals by using pesticides.

However, in the cited papers, the function $k$ is either a constant or a strictly positive function depending on $a$.  In such cases, it is well known from the general theory that all nontrivial solutions of the corresponding system (commonly named Lotka-McKendrick systems) are asymptotically exponentially growing or decaying, according to the size of a certain biological quantity (the so called net reproduction rate), see \cite{An} and also \cite{fmvM3} for related results concerning  time-independent steady states.

In this paper we are not interested in large time controllability, i.e. asymptotic behavior of the solution of \eqref{1}, but we want to address the problem of null controllability at each fixed time $T>0$. More precisely, we will give sufficient conditions so that, for all initial data $y_0$ in a suitable space, there exists a control $f$ that brings the solution $y$ of \eqref{1} at time $T$ at zero, i.e.
\[
y(T,a,x) =0\]
for all $x\in(0,1)$ and all $a$ in  a suitable subdomain of $(0,A)$.

Our study has obviosly many connections with related ones for the heat equation. Let us recall that the null controllability for linear parabolic equations has been extensively studied in the last years using Carleman inequalities and duality argument, not only when $k$ is a constant (see, for example, \cite{lr}, \cite{LRL}), but also when $k$ degenerates at the boundary of the space domain (see, for example, \cite{acf}, \cite{cfr}-\cite{cmv2}, \cite{e}, \cite{fl}) or in the interior (see, for example, \cite{bcg}, \cite{bfm1}, \cite{f}, \cite{JDE}, \cite{fm}-\cite{fm2}). As far as we know, the first controllability result for an age population dynamics model is established in \cite{ainsebalanglais}, where the authors proved that a set of profiles is approximately reachable. Later, in \cite{ainsebaanita} a local exact controllability was proved. In particular, the authors showed that, if the initial distribution is small enough, one can find a control that leads the population to extinction (see also \cite{ainsebaiannelli} and \cite{iannelli}).
Null controllability is also studied for {\it nonlinear} population dynamics models, see \cite{ainsebaiannelli} and \cite{traore}: in the first paper the authors  studied the controllability of nonlinear diffusive dynamic populations when the fertility and the mortality rates depend on the total population; in the second one, the authors considered a nonlinear distribution of newborns of the form $F(\int_0^A \beta (t,a, x)y (t, a, x) da)$.
However, in all the previous papers the dispersion coefficient {\it $k$ is a constant or a strictly positive function}.

To our best knowledge, \cite{aem} is the first paper where  the dispersion coefficient, which depends on the space variable $x$, can degenerate.  In particular, the authors assume that $k$ degenerates  at the boundary (for example $k(x) = x^\alpha,$  being $x \in (0,1)$ and $\alpha >0$). Using Carleman estimates for the adjoint problem, the authors prove null controllability for \eqref{1} under the condition $T\ge A$. However, this assumption is not realistic when $A$ is too large. To overcome this problem in \cite{em}, the authors used Carleman estimates and a fixed point method via the Leray - Schauder Theorem. However, while  in \cite{aem} and in \cite{em}, the degenerating operator is in {\it divergence form} - shortly (Df) -, i.e. $(k(x)y_x)_x$, in this paper we consider the degenerating operator in {\it nondivergence form} - shortly (NDf) - and we allow the function $k$ to degenerate not only at the boundary, but also in the interior of the state space. Observe that, in the case of a boundary degeneracy, we cannot derive the null controllabilility 
for \eqref{1} by the one of the problems in divergence form. Indeed, it is proved in \cite{cfr} that in this situation, i.e. when the degeneracy is at the boundary of the domain, and when the functions are independent of $a$ (i.e. if we have the degenerate heat equation), the equation of \eqref{1} can be rewritten as
\begin{equation}\label{divnew}
\frac{\partial y}{\partial t}  
-(k(x)y_{x})_x+ k_x(x)y_x+\mu(t, x)y =f(t,x)\chi_{\omega} \end{equation}
at the price of adding the drift term $k_x(x)y_x$. Such an addition has major consequences: as described in \cite{cmv0}, degenerate equations of the form \eqref{divnew} are well posed in $L^2(0,1)$ under the structural assumption
\[
k_x(x) \le C\sqrt{k(x)},
\]
for a strictly positive constant $C$. Imposing this condition on $k_x$, for $k(x)=x^\alpha$, gives $\alpha \ge $2. This necessary condition that ensures the well posedness of \eqref{1} makes it not null controllable (see \cite{fm1} for the interior degeneracy).
For this reason, in this paper as in \cite{cfr}, \cite{cfr1}, \cite{f}, \cite{JDE} or \cite{fm1}, we prove null controllability for \eqref{1} without deducing it by the previous results for the problem in divergence form. Therefore, this paper complements \cite{aem}. Indeed, we do not require as in \cite{aem}, that $T \ge A$, but $T <A$ (see Hypothesis \ref{conditionbeta}).  Clearly, this assumption is more interesting, since it is reasonable to control the population  in small times and this is important if $y$ represents, for example, a demaging insect population or a pest population. Moreover,  while in \cite{em}  the authors  used Carleman estimates and a generalization of the Leray - Schauder fixed point Theorem and the multi-valued theory, here we use only Carleman estimates for the non degenerate and the degenerate problem,  and a  technique based on cut off functions, making the proof slimmer and easier to read. Last but not the least, we underline that in \cite{aem} and in \cite{em} only the case of a boundary degeneracy is considered. 
If the function $k$ in \eqref{1} degenerates in the interior of $(0,1)$ and the problem is in divergence form,  related results can be founded in \cite{idriss}. To our best knowledge, as written before, this is the first paper where  the problem in {\it nondivergence form} is considered allowing the diffusion coefficient to degenerate at the boundary or  in the interior of $(0,1)$ (when $y$ is independent of $a$ we refer, for example, to \cite{fm1}). We underline that in \cite{idriss} the authors assume that, if $x_0 \in (0,1)$ is the degenerate point, the function $k\in
C[0,1] \cap C^1([0,1]\setminus \{x_0\})$; moreover, they require the existence of a constant $M \in [0,1)$ such that
$(x-x_0)k' \le M k$ a.e. in $[0,1]$. In this paper, we consider a less regular function $k$ and we allow the constant $M$ to approach $2$, i.e. $M \in [0,2)$, considering the so-called strongly degenerate case.

The paper is organized in the following way: in Section \ref{sec3-1} we study the well posedness of the problem in the case that the dispersion coefficient $k$ degenerates either at the boundary or in the interior of the state space. Section  \ref{sec-3} is divided into three subsections: in the first one we deduce a Carleman estimate for the non degenerate problem in nondivergence form by a Carleman estimate for the non degenerate  problem in divergence form (for the reader's convenience, we give its proof in the Appendix); the second and the third subsections are devoted to study Carleman estimates in the case that $k$ degenerates at the boundary of the state space or in its interior, respectively. Finally, in Section \ref{osservabilita} we prove null controllability via a null controllability result for an intermediate system, observality inequalities and cut off functions. 

A final comment on the notation: by $c$ or $C$ we shall denote
{\em universal} strictly positive constants, which are allowed to vary from line to
line.

\section{Well posedness result}\label{sec3-1}

To study  well posedness we assume that the dispersion coefficient  $k$ satisfies one of the following assumptions:
\begin{Assumptions}\label{BAss} {\bf Boundary degenerate case (BD):} 
\[
k \in  C([0,1])\quad
 k>0 \text{ in }(0,1) \text{ and } k(0)=0 \text{ or } k(1)=0.
 \]
 \end{Assumptions}
 \begin{Assumptions}\label{Ass0}
{\bf Interior weakly degenerate case (IWD):} There exists $x_0 \in (0,1)$ such
that $k(x_0)=0$, $k>0$ on $[0, 1]\setminus \{x_0\}$, $k\in
W^{1,1}(0,1)$ and there exists $M\in (0,1)$ such that $(x-x_0)k'
\le Mk$ a.e. in $[0,1]$.
\end{Assumptions}

\begin{Assumptions}\label{Ass01}
{\bf Interior strongly degenerate case (ISD):} There exists $x_0 \in (0,1)$
such that $k(x_0)=0$, $k>0$ on $[0, 1]\setminus \{x_0\}$, $k\in
W^{1, \infty}(0,1)$ and there exists $M \in [1,2)$ such that
$(x-x_0)k' \le M k$ a.e. in $[0,1]$.
\end{Assumptions}

Thus, we assume that the function $k$ can degenerate at the boundary of the domain or at an interior point; for example, as $k$ one can consider $k(x)=x^\alpha$, $k(x)=(1-x)^\alpha$ or $k(x)= |x-x_0|^\alpha$, $\alpha >0$.
\vspace{0.5cm}

On the rates $\mu$ and $\beta$ we assume:
\begin{Assumptions}\label{ratesAss}
 The  functions $\mu$ and $\beta$ are such that
\begin{equation}\label{3}
\begin{aligned}
&\bullet \beta \in C(\bar Q_{A,1}) \text{ and } \beta \geq0  \text{ in } Q_{A,1}, \\
&\bullet \mu \in C(\bar Q) \text{ and }  \mu\geq0\text{ in } Q.
\end{aligned}
\end{equation}
\end{Assumptions}

To prove well posedness  of  \eqref{1}, we introduce, as in \cite{cfr} or in \cite{cfr1}, the following
weighted Lebesgue and Hilbert spaces
\begin{equation}\nonumber
 L^2_{\frac{1}{k}}(0,1) :=\left\{ u \in L^2(0,1) \
\Big|  \int_0^1u^2\frac{1}{k}dx<\infty \right\},
\end{equation}
\begin{equation}\label{578}
\qquad H^1_{\frac{1}{k}}(0,1)
:=L^2_{\frac{1}{k}}(0,1)\cap H^1_0(0,1)
\end{equation}
and
\begin{equation}\label{579}
H^2_{\frac{1}{k}}(0,1) :=\Big\{ u \in H^1_{\frac{1}{k}}(0,1)\; \big|
\;ku_{xx} \in L^2_{\frac{1}{k}}(0,1)\Big\},
\end{equation}
 in the boundary degenerate case; while in the interior degenerate case,  as in \cite{fm1}, we consider, in place of $H^2_{\frac{1}{k}}(0,1)$, the space
\[
H^2_{\frac{1}{k}, x_0}(0,1) :=\Big\{ u \in H^1_{\frac{1}{k}}(0,1) \,
\big| \, u'\in H^1(0,1)\Big\},
\]
that can be written  in
a more appealing way as
\[
H^2_{\frac{1}{k}, x_0}(0,1) :=\Big\{ u \in H^1_{\frac{1}{k}}(0,1) \,
\big| \, u'\in H^1(0,1) \mbox{ and }ku'' \in
L^2_{\frac{1}{k}}(0,1)\Big\}.
\]
In every case, we consider the following norms
\[
\|u\|_{{\frac{1}{k}}}^2:= \int_0^1u^2\frac{1}{k}dx,
\]
\[
\|u\|_{1,{\frac{1}{k}}}^2 := \int_0^1u^2\frac{1}{k}dx + \int_0^1u_x^2dx
\]
and
\[
\|u\|_{2,{\frac{1}{k}}}^2 :=\|u\|_{1,{\frac{1}{k}}}^2 + \int_0^1ku^2_{xx}dx.
\]
Observe that, if $k$ is nondegenerate, the spaces $L^2_{\frac{1}{k}}(0,1)$, $H^1_{\frac{1}{k}}(0,1)$ and $H^2_{\frac{1}{k}}(0,1)$ (or $H^2_{\frac{1}{k}, x_0}(0,1)$) coincide, respectively, with $L^2(0,1)$, $H^1_0(0,1)$ and $H^2(0,1)\cap H^1_0(0,1).$

Denoting by
$
\cH^2_{{\frac{1}{k}}}(0,1) 
$ the space $
H^2_{{\frac{1}{k}}}(0,1) 
$  or $
H^2_{{\frac{1}{k}, x_0}}(0,1) ,
$ 
we have, as in \cite{cfr}, \cite{cfr1} or \cite{fm1}, that the  operator
\[\mathcal A_0u:= ku_{xx},\qquad    D(\mathcal A_0): = \cH^2_{{\frac{1}{k}}}(0,1)\]
is self--adjoint, nonpositive  and generates an analytic contraction
semigroup of angle $\pi/2$ on the space $L^2_{{\frac{1}{k}}}(0,1)$.

Now, setting $ \mathcal A_a u := \ds \frac{\partial  u}{\partial a}$, we have that
\[
\mathcal Au:= \mathcal A_a u - 
\mathcal A_0 u,
\]
for 
\[
u \in D(\mathcal A) =\left\{u \in L^2(0,A;D(\mathcal A_0)) : \frac{\partial u}{\partial a} \in  L^2(0,A;H^1_{\frac{1}{k}}(0,1)), u(0, x)= \int_0^A \beta(a, x) u(a, x) da\right\},
\]
generates a strongly continuous semigroup on $L^2(0,A) \times L^2 _{{\frac{1}{k}}}(0,1)$ (see also \cite{iannelli}). Moreover, the operator $B(t)$ defined as
\[
B(t) u:= -\mu(t,a,x) u,
\]
for $u \in D(\mathcal A)$, can be seen as a bounded perturbation of $\mathcal A$ (see, for example, \cite{acf}); thus also
$
(\mathcal A + B(t), D(\mathcal A))
$ generates a strongly continuous semigroup.

Setting $L^2_{{\frac{1}{k}}}(Q):= L^2(Q_{T,A}; L^2_{\frac{1}{k}}(0,1))$ and $L^2_{{\frac{1}{k}}}(Q_{A,1}):= L^2(0,A;L^2_{\frac{1}{k}}(0,1))$, the following well posedness  result holds (see \cite{en}, \cite{lm}):
\begin{Theorem}\label{theorem_nondivergence}
Assume that Hypotheses $\ref{ratesAss}$ and one among Hypothesis $\ref{BAss}$ - $\ref{Ass01}$ are satisfied.  For all $f \in
L^2_{{\frac{1}{k}}}(Q)$ and $y_0 \in L^2_{{\frac{1}{k}}}(Q_{A,1})$, the system \eqref{1} admits a unique solution 
$y \in \mathcal U:= C\big([0,T];
L^2_{{\frac{1}{k}}}(Q_{A,1})\big) \cap L^2 \big(0,T;
H^1(0,A; H^1_{\frac{1}{k}}(0,1)\big)$.
In addition, if $f\equiv 0$, 
$
u\in C^1\big([0,T];L^2_{{\frac{1}{k}}}(Q_{A,1})\big).
$ 
\end{Theorem}
\section{Carleman estimates}\label{sec-3}

In this section we show
Carleman estimates for the following system:
\begin{equation} \label{adjoint}
\begin{cases}
\ds \frac{\partial z}{\partial t} + \frac{\partial z}{\partial a}
+k(x)z_{xx}-\mu(t, a, x)z =f ,& (t,a,x) \in Q,\\
  z(t, a, 0)=z(t, a, 1)=0, & (t,a) \in Q_{T,A},\\
   z(t,A,x)=0, & (t,x) \in Q_{T,1},
\end{cases}
\end{equation}
where the function $k$ is non degenerate (this will be crucial for the following) or satisfies one of Hypothesis \ref{BAss}, \ref{Ass0} or \ref{Ass01}. 
\paragraph{Carleman inequalities in the non degenerate case}

First of all assume that $k$ is non degenerate. Then, the following estimate holds:
\begin{Theorem}\label{nondegenere}
Let $z\in 
\mathcal{V}:=L^2\big(Q_{T,A}; H^2(0,1)\cap H^1_0(0,1)\big) \cap H^1(0,T; H^1(0,A;H^1_0(0,1)))$ be the solution of
\eqref{adjoint}
where $f \in L^{2}(Q)$ and $k \in C^{1}([0,1])$ is a strictly
positive function. Then, there exist two strictly positive constants $C$ and $s_0$,
such that, for any $s\geq s_0$, $z$ satisfies the estimate
\begin{equation} \label{570}
\begin{aligned}
&\int_{Q}(s^{3}\phi^{3}z^{2}+s\phi z_{x}^{2})e^{2s\Phi} dxdadt \leq C \Big(\int_{Q}f^{2}e^{2s\Phi}dxdadt
 -
s\kappa\int_0^T\int_0^A\left[ke^{2s\Phi}\phi(z_x)^2
\right]_{x=0}^{x=1}dadt\Big).
\end{aligned}
\end{equation}
Here the functions $\phi$ and $\Phi$ are
defined as follows
\begin{equation}\label{571}
\begin{gathered}
\phi(t,a,x)=\Theta(t,a)e^{\kappa\sigma(x)}, \quad
\Theta(t, a)= \frac{1}{t^{4}(T-t)^{4}a^{4}},\\
\Phi(a,t,x)=\Theta(t,a)\Psi(x), \quad
\Psi(x)=e^{\kappa\sigma(x)}-e^{2\kappa\|\sigma\|_{\infty}},
\end{gathered}
\end{equation}
$(t,a,x)\in Q$, $\kappa>0$ and $\sigma (x) :=\mathfrak{d}\int_x^1\frac{1}{k(t)}dt$, where $\fd=\|k'\|_{L^\infty(0,1)}$.
\end{Theorem}

The proof of the previous result is based on the next Carleman estimate which is proved in the Appendix.
\begin{Theorem}\label{nondegenerediv}
Let $z \in \mathcal V$ be the solution of
\begin{equation} \label{adjointdiv}
\begin{cases}
\ds \frac{\partial z}{\partial t} + \frac{\partial z}{\partial a}
+(k(x)z_x)_x-\mu(t, a, x)z =f ,& (t,x,a) \in Q,\\
  z(t, a, 0)=z(t, a, 1)=0, & (t,a) \in Q_{T,A},\\
   z(t,A,x)=0, & (t,x) \in Q_{T,1},
  \end{cases}
\end{equation}
where $f$ and $k$ are as in the previous theorem. Then, there exist two strictly positive constants $C$ and $s_0$,
such that, for any $s\geq s_0$, $z$ satisfies the estimate
\begin{equation} \label{570'}
\begin{aligned}
&\int_{Q}(s^{3}\phi^{3}z^{2}+s\phi z_{x}^{2})e^{2s\Phi} dxdadt \leq C \Big(\int_{Q}f^{2}e^{2s\Phi}dxdadt
  -
s\kappa\int_0^T\int_0^A\left[ke^{2s\Phi}\phi(z_x)^2
\right]_{x=0}^{x=1}dadt\Big),
\end{aligned}
\end{equation}
with  $\phi$ and $\Phi$  defined as in \eqref{571}.
\end{Theorem}
\begin{proof}[Proof of Theorem \ref{nondegenere}]
Rewrite the equation of
    \eqref{adjoint} as $ \ds \frac{\partial z}{\partial t} + \frac{\partial z}{\partial a}
+(k(x)z_x)_x-\mu(t, a, x)z = \bar{f}, $ where $\bar{f}
    := f + k'z_x$. Then, applying Theorem \ref{nondegenerediv}, there exist
    two strictly positive constants $C$ and $s_0 >0$, such that,
 for all $s \ge s_0$,
\begin{equation}\label{sd}
\begin{aligned}
&\int_{Q}(s^{3}\phi^{3}z^{2}+s\phi z_{x}^{2})e^{2s\Phi} dxdadt \leq C \Big(\int_{Q}\bar f^{2}e^{2s\Phi}dxdadt
 -
s\kappa\int_0^T\int_0^A\left[ke^{2s\Phi}\phi(z_x)^2
\right]_{x=0}^{x=1}dadt \Big).
\end{aligned}
\end{equation}
Using the definition of $\bar{f}$, the term $\int_{Q}\bar{f}^2
e^{2s\Phi(t,x)}dxdadt$ can be estimated in the following way
\begin{equation}\label{sd1}
\begin{aligned}
\int_Q \bar{f}^2e^{2s\Phi}dxdadt &\le
2\int_Q f^2e^{2s\Phi}dxdadt
+2\|k'\|_{L^\infty(0,1)}^2\int_Q
e^{2s\Phi}(z_x)^2dxdadt\\
& \le  2\int_Qf^2e^{2s\Phi}dxdadt
+2\|k'\|_{L^\infty(0,1)}^2c\int_Q \Theta e^{\kappa\sigma}
e^{2s\Phi}(z_x)^2dxdadt,
\end{aligned}
\end{equation}
where $c:=A^4 \max_{[0, T]}(t(T-t))^4= A^4\displaystyle\left(\frac{T}{2}\right)^8$. Thus, by \eqref{sd} and \eqref{sd1}, one has
\[
\begin{aligned}
    &\int_Q\left(s^{3}\phi^{3}z^{2}+s\phi z_{x}^{2} - 2\|k'\|_{L^\infty(0,1)}^2c\phi
   z_x^2\right)e^{2s\Phi}dxdadt\\
&\le C\left(\int_Qf^{2}e^{2s\Phi}dxdadt   -
s\kappa\int_0^T\int_0^A\left[ke^{2s\Phi}\phi(z_x)^2
\right]_{x=0}^{x=1}dadt\right).
    \end{aligned}
    \]
    Now, let $s_1>0$ be such that $\displaystyle \frac{s_1}{2} \ge 2\|k'\|_{L^\infty(0,1)}^2c$. Then, for all $s \ge s_1$
    \[
     \begin{aligned}
   & \int_Q \left(s\phi z_x^2 - 2\|k'\|_{L^\infty(0,1)}^2c \phi
   z_x^2 \right)e^{2s\Phi}dxdadt\ge \frac{s}{2} \int_Q \phi z_x^2e^{2s\Phi}dxdadt.
    \end{aligned}  \]
    Hence the claim follows for all $s \ge \max\{s_0, s_1\}$.
\end{proof}
Actually we can prove Theorem \ref{nondegenere} directly, but we have to assume on $k$ more regularity, for example $k\in C^2[0,1]$ or, at least, $k \in W^{2, \infty}(0,1)$. Indeed, in this case, we have to estimate an integral containing the term $(k \Phi_{xx})_x$.

\begin{Remark}
The previous Theorems still hold under the weaker assumption $k \in W^{1, \infty}(0,1)$ without any additional assumption. 
\\
On the other hand, if we require $k \in W^{1,1}(0,1)$ then we have to add the following hypothesis:
{\it  there exist two functions $\fg \in L^1(0,1)$,
$\fh \in W^{1,\infty}(0,1)$ and two strictly positive constants
$\fg_0$, $\fh_0$ such that $\fg(x) \ge \fg_0$ and
\[-\frac{k'(x)}{2\sqrt{k(x)}}\left(\int_x^1\fg(t) dt + \fh_0 \right)+ \sqrt{k(x)}\fg(x) =\fh(x)\quad \text{for a.e.} \; x \in [0,1],\]
in the divergence case,
\[\frac{k'(x)}{2\sqrt{k(x)}}\left(\int_x^1\fg(t) dt + \fh_0 \right)+ \sqrt{k(x)}\fg(x) =\fh(x)\quad \text{for a.e.} \; x \in [0,1],\]
in the nondivergence one.
}
\\
In this case, i.e. if $k \in W^{1,1}(0,1)$,  the function $\Psi$ in \eqref{571} becomes
\begin{equation}\label{Psi_new}
\Psi(x):= - r\left[\int_0^x
\frac{1}{\sqrt{k(t)}} \int_t^1
\fg(s) dsdt + \int_0^x \frac{\fh_0}{\sqrt{k(t)}}dt\right] -\mathfrak{c}, 
\end{equation}
where $r$ and $\mathfrak{c}$ are suitable strictly positive functions.
\end{Remark}
Thus we have the next theorem
\begin{Assumptions}\label{ipoadebole}$\ $\\
\begin{itemize}
\item[$(a_1)$] $k\in W^{1,1}(0,1)$, and there exist two functions $\fg \in L^1(0,1)$,
$\fh \in W^{1,\infty}(0,1)$ and two strictly positive constants
$\fg_0$, $\fh_0$ such that $\fg(x) \ge \fg_0$ and
\[-\frac{k'(x)}{2\sqrt{k(x)}}\left(\int_x^1\fg(t) dt + \fh_0 \right)+ \sqrt{k(x)}\fg(x) =\fh(x)\quad \text{for a.e.} \; x \in [0,1],\]
in the (Df) case,
\[\frac{k'(x)}{2\sqrt{k(x)}}\left(\int_x^1\fg(t) dt + \fh_0 \right)+ \sqrt{k(x)}\fg(x) =\fh(x)\quad \text{for a.e.} \; x \in [0,1],\]
in the (NDf) one,
or\\
\item[$(a_2)$] $k\in W^{1,\infty}(0,1)$.
\end{itemize}
\end{Assumptions}
Define $\Phi(t,a,x)$, $\phi(t,a,x)$, $ \Theta(t,a)$ and $\sigma$ as in \eqref{571} and
\begin{equation}\label{c_1NEW}
\Psi(x):=\begin{cases} \displaystyle - r\left[\int_0^x
\frac{1}{\sqrt{k(t)}} \int_t^1
\fg(s) dsdt + \int_0^x \frac{\fh_0}{\sqrt{k(t)}}dt\right] -\mathfrak{c}, &\text{ if } (a_1) \text{ holds,}\\
\displaystyle  e^{r\sigma(x)}-\mathfrak{c}, &\text{ if } (a_2) \text{
holds,}\end{cases}
\end{equation}
where $r>0$ and $\mathfrak{c} >0$ is chosen in the second case in such a way that $\max_{[0,1]}\Psi <0$.
\begin{Theorem}\label{nondegenereDebole}
Assume that Hypothesis \ref{ipoadebole} is satisfied. Let $z\in 
\mathcal{V}$ be the solution of
\eqref{adjoint} or of \eqref{adjointdiv}
where $f \in L^{2}(Q)$. Then, there exist two  strictly positive constants $C$ and $s_0$,
such that, for any $s\geq s_0$, $z$ satisfies the estimate
\begin{equation} \label{570New}
\begin{aligned}
&\int_Q \left(s\Theta (z_x)^2 + s^3 \Theta^3
 z^2\right)e^{2s\Phi}dxdadt\le C\left(\int_Q f^{2}e^{2s\Phi}dxdadt -
(B.T.)\right),
\end{aligned}
\end{equation}
where
\[
(B.T.):= \begin{cases}
sr\int_0^T\int_0^A\Theta(t)\left[\sqrt{k}\left(\int_x^1 \fg(\tau) d\tau +
\fh_0 \right)(z_x)^2 e^{2s\Phi}\right]_{x=0}^{x=1}dadt, & \text{in the (NDf)},\\
\quad\\
sr \int_0^{T}\int_0^A
\left[k^{3/2}e^{2s\Phi}\Theta \left(\int_x^1 \fg(\tau) d\tau + \fh_0
\right) (z_x)^2\right]^{x=1}_{x=0}dadt, &  \text{in the (Df)},
\end{cases}
\]
if $(a_1)$ holds and
\begin{equation} \label{570Newnondiv}
\begin{aligned}
&\int_Q \left(s\Theta e^{r\sigma} (z_x)^2 + s^3 \Theta^3
e^{3r\sigma} z^2\right)e^{2s\Phi}dxdadt\le C\left(\int_Q f^{2}e^{2s\Phi}dxdadt -
(B.T.)\right),
\end{aligned}
\end{equation}
where
$
(B.T.):=
sr\int_0^T\int_0^A\left[ke^{2s\Phi}\Theta e^{r\sigma}(z_x)^2
\right]_{x=0}^{x=1}dadt,
$
if $(a_2)$ is in force.
\end{Theorem}
(See the Appendix for the proof.)

\paragraph{Carleman inequalities when the degeneracy is at the boundary}\label{sec-3-1}

In this subsection we will consider the case when $k(0)=0$ or $k(1)=0$. In both cases we assume that $\mu$ satisfies \eqref{3}. On the other hand, on $k$ we make different assumptions:
\begin{Assumptions}\label{BAss01} The function
$k\in C^0[0,1]\bigcap C^2(0,1]$  is such that $k(0)=0$, $k>0$ on
$(0,1]$ and there exist $\varepsilon\in (0,1]$ and $M\in (0,2)$ such that
the function $\displaystyle\frac{xk_x}{k(x)}$
$\in \:L^{\infty}(0,\varepsilon)$,
\small$\displaystyle\frac {xk_x(x)}{k(x)} \le M $ and
$\displaystyle \left( \frac{xk_x(x)}{k(x)}\right)_{x} \in L^{\infty}(0,\varepsilon)$.
\end{Assumptions}
\begin{Assumptions}\label{BAss02} The function
$k\in C^0[0,1]\bigcap C^2[0,1)$  is such that $k(1)=0$, $k>0$ on
$(0,1)$ and there exist $\varepsilon\in (0,1]$ and $M\in (0,2)$ such that
the function \small$\displaystyle\frac{(x-1)k_x}{k(x)}$\normalsize\:
$\in L^{\infty}(1-\varepsilon,1)$,
 \small$\displaystyle\frac {(x-1)k_x(x)}{k(x)} \le M $
and $\displaystyle\left( \frac{(x-1)k_x(x)}{k(x)}\right)_{x}\in L^{\infty}(1-\varepsilon,1)$.
\end{Assumptions}

Now, let us introduce the weight functions

\begin{equation}\label{13}
\varphi(t,a,x):=\Theta(t,a)(p(x) - 2 \|p\|_{L^\infty(0,1)}),
\end{equation}
and
\begin{equation}\label{13'}
\bar\varphi(t,a,x):=\Theta(t,a)(\bar p(x) - 2 \|\bar p\|_{L^\infty(0,1)}),
\end{equation}
where $\Theta$ is as in \eqref{571},
$\displaystyle p(x):=\int_0^x\frac{y}{k(y)}~e^{Ry^2}dy$ and $
\displaystyle \bar p(x):=\int_{0}^x\frac{y-1}{k(y)}~e^{R(y-1)^2}dy$, with  $R>0,$ if $k$ satisfies Hypothesis \ref{BAss01} or Hypothesis \ref{BAss02}, respectively.
Observe that $ \varphi(t,a, x) ,  \bar\varphi(t,a, x) <0$ for all $(t,x) \in Q$ and
$\varphi(t,a, x), \bar  \varphi(t,a, x)  \rightarrow - \infty \, \text{ as } t \rightarrow
0^+, T^-$  or  $a \rightarrow
0^+$.  The following estimates hold:

\begin{Theorem}\label{Cor1}
Assume that Hypothesis $\ref{BAss01}$ is satisfied for some
$\varepsilon\in (0,1]$. Then,
there exist two strictly positive constants $C$ and $s_0$ such that every
solution $v$ of \eqref{adjoint} in
\[
\mathcal{V}_1:=L^2\big(Q_{T,A}; H^2_{{\frac{1}{k}}}(0,1)\big) \cap H^1\big(0, T; H^1(0,A;H^1_{{\frac{1}{k}}}(0,1))\big)\]
satisfies, for all $s \ge s_0$,
\[
\begin{aligned}
\int_{Q}\left(s \Theta v_x^2
                + s^3\Theta^3\text{\small$\displaystyle\Big(\frac{x}{k}\Big)^2$\normalsize}
                  v^2\right)e^{2s\varphi}dxdadt
&\le
C\int_{Q}f^{2}\text{\small$\frac{e^{2s\varphi}}{k}$\normalsize}dxdadt
\\&+ s\,C\int_0^T \int_0^A\Theta(t,a)\Big[ x v_x^2 e^{2s\varphi}\Big](t, a, 1)dadt.
\end{aligned}\]
\end{Theorem}

\begin{Theorem}\label{Cor1'}
Assume that Hypothesis $\ref{BAss02}$ is satisfied for some
$\varepsilon\in (0,1]$. Then,
there exist two strictly  positive constants $C$ and $s_0$ such that every
solution $v$ of \eqref{adjoint} in $\mathcal V_1$
satisfies, for all $s \ge s_0$,
\[
\begin{aligned}
\int_{Q}\left(s\Theta v_x^2
                  + s^3 \Theta^3\text{\small$\displaystyle\Big(\frac{x-1}{k}\Big)^2$\normalsize}
                    v^2\right)e^{2s\bar\varphi}dxdadt
&\le
C\int_{Q}f^{2}\text{\small$\frac{~e^{2s\bar \varphi}}{k}$\normalsize}dxdadt
\\&+sC \int_0^T\int_0^A\!\!\! \Theta(t,a) \Big[ (1-x) v_x^2 e^{2s\bar\varphi}\Big](t, a,0)dadt.
\end{aligned}\]
\end{Theorem}
Clearly the previous  Carleman estimates hold
for every function $v$ that satisfies \eqref{adjoint} in $(0,T)\times(0,A)\times (0,B)$ or $(0,T)\times(0,A)\times(B,1)$ as long as $(0,1)$ is substituted by $(0,B)$ or $(B,1)$ and $k$ satisfies Hypothesis \ref{BAss01} in $(0,B)$ or Hypothesis \ref{BAss02} in $(B,1)$, respectively.

\begin{Remark}
Observe that Theorems \ref{Cor1} and \ref{Cor1'} improve \cite[Theorems 3.3. and 3.4]{cfr} and \cite[Theorem 3 and 4]{cfr1}. Indeed,  here we assume that $k$ is of class $C^2(0,1]$ (or $C^2[0,1)$) and not $C^3(0,1]$ (or $C^3[0,1)$) as therein, where $y$ was independent of $a$.
\end{Remark}
\vspace{0.4cm}

In the following, we will prove only Theorem  \ref{Cor1} since the proof of Theorem
\ref{Cor1'} is analogous.

\paragraph{Proof of Theorem \ref{Cor1}}
As a first step assume that $\mu \equiv 0$.

In order to prove Theorem \ref{Cor1},  we define, for
$s > 0$, the function
\[
w(t,a,x) := e^{s \varphi (t,a,x)}v(t,a,x)
\]
where $v$ is the solution of \eqref{adjoint} in $\mathcal{V}_1$; observe that,
since $v\in\mathcal{V}_1$, $w\in\mathcal{V}_1$. Clearly, one has that $w$ satisfies
\begin{equation}\label{1'}
\begin{cases}
(e^{-s\varphi}w)_t + (e^{-s\varphi}w)_a +k(x)(e^{-s\varphi}w)_{xx}  =f(t,a,x), & (t,x) \in
Q,
\\[5pt]
w(0, a, x)= w(T,a, x)= 0, & (a,x) \in Q_{A,1},
\\[5pt]
w(t,A,x)=w(t,0,x)=0, & (t,x) \in  Q_{T,1},
\\[5pt]
w(t, a,0)= w(t, a, 1)= 0, & (t,a) \in  Q_{T,A}.
\end{cases}
\end{equation}
Defining $Lw:= w_t + w_a+kw_{xx}$ and $L_sw:=
e^{s\varphi}L(e^{-s\varphi}w)$, the equation of \eqref{1'} can be
recast as follows
\begin{center}
$
L_sw =L^+_sw + L^-_sw=e^{s\varphi}f,
$
\end{center}
where
\[
\begin{cases}
L^+_sw := kw_{xx}
 - s (\varphi_t+ \varphi_a) w + s^2k \varphi_x^2 w,
\\[5pt]
L^-_sw := w_t + w_a-2sk\varphi_x w_x -
 sk\varphi_{xx}w.
 \end{cases}
\]
Moreover, set $<u, v>_{L^2_{\frac{1}{k}}(Q)} \::=\;
$\small$\displaystyle\int_{Q}uv\frac{1}{k}dxdadt$\normalsize,
\;one has
\begin{equation} \label{Dis1}
\|L^+_sw \|^2_{L^2_{\frac{1}{k}}(Q)} + \:\|L^-_sw\|^2_{L^2_{\frac{1}{k}}(Q)} +
2<L^+_sw, L^-_sw>_{L^2_{\frac{1}{k}}(Q)} =\;
\|fe^{s\varphi}\|^2_{L^2_{\frac{1}{k}}(Q)}.
\end{equation}
Now, we compute the inner product  $<L^+_sw, L^-_sw>_{L^2_{\frac{1}{k}}(Q)}$ whose first expression is given in the following lemma
\begin{Lemma}\label{lemma1}Assume Hypothesis $\ref{BAss01}$.
The following identity holds
\begin{equation}\label{D&BT}
\left.
\begin{aligned}
<L^+_sw,L^-_sw>_{L^2_{\frac{1}{k}}(Q)}
\;&=\;
s\int_{Q}(k\varphi_{xx}+(k\varphi_x)_x)w_x^2dxdadt
\\[3pt]&+ s^3 \int_{Q}\varphi_x^2(k\varphi_{xx}+(k\varphi_x)_x)w^2dxdadt
\\[3pt]
&
-2 s^2\int_{Q} \varphi_x\varphi_{xt}w^2dxdadt
 +\frac{s}{2}\int_{Q} \frac{\varphi_{tt}+ \varphi_{aa}}{k}w^2dxdadt
\\[3pt] &+s\int_{Q_T}(k\varphi_{xx})_{x}ww_x dxdadt\\
&
+s\int_{Q} \frac{\varphi_{ta}}{k}w^2dxdadt
-2 s^2\int_{Q} \varphi_x\varphi_{xa}w^2dxdadt
\end{aligned}\right\}\;\text{\{D.T.\}}
\end{equation}
\begin{equation}\nonumber
\hspace{55pt}
\text{\{B.T.\}}\;\left\{
\begin{aligned}
&
-\frac{1}{2}\int_0^A\int_0^{1}\Big[w_x^2\Big]_0^Tdxda
+\int_0^T\int_0^A\Big[w_x(w_t+ w_a)\Big]_0^{1}dadt
\\[3pt]&
- s \int_0^T\int_0^A\Big[ k \varphi_xw_x^2\Big]_0^{1}dadt\\[3pt]&
- s \int_0^T\int_0^A\!\!\Big[ k\varphi_{xx}ww_x\Big]_0^{1}dadt
+\frac{1}{2}\int_0^A \int_{0}^{1}\!\!\Big[ (s^2 \varphi_x^2 - s \frac{\varphi_t+\varphi_a}{k})w^2\Big]_0^Tdxda
\\[3pt]
&
- s \int_0^T\int_0^A\Big[  (s^2 k\varphi_x^3- s\varphi_x\varphi_t- s\varphi_x\varphi_a)w^2\Big]_0^{1}dadt\\
&
 -\frac{1}{2}\int_0^T\int_0^{1} \big[w_x^2\big]_0^Adxdt
 +\frac{1}{2}\int_0^T\int_0^{1}\big[ \big(s^2 \varphi_x^2 - s \frac{\varphi_t+\varphi_a}{k} \big)w^2\big]_0^Adxdt.
\end{aligned}\right.
\end{equation}
\end{Lemma}
\vspace{5pt}
\begin{proof}
It results, integrating by parts,
\[
<L^+_sw,L^-_sw>_{L^2_{\frac{1}{k}}(Q)}= I_1+ I_2+ I_3+ I_4,
\]
where
\[
I_1= \int_{Q}w_{xx}(w_t -2sk\varphi_x w_x -
 sk\varphi_{xx}w) dxdadt,
\]
\[
I_2=\int_{Q}\frac{1}{k}\big(- s \varphi_t w
  +s^2k\varphi_x^2 w\big)(w_t -2sk\varphi_x w_x -
 sk\varphi_{xx}w) dxdadt,
\]
\[
I_3= \int_{Q}(w_{xx}-s \frac{(\varphi_t +\varphi_a) }{k} w+ s^2\varphi_x^2w)w_a dxdadt
\]
and
\[
I_4= -s\int_{Q}\frac{\varphi_a w}{k}  (w_t -2sk\varphi_x w_x -
 sk\varphi_{xx}w) dxdadt.
\]
By several integrations by parts in space and in time (see \cite{cfr} or \cite{cfr1}), we get
\begin{equation}\label{!1}
\begin{aligned}
I_1 + I_2&= s\int_{Q}(k\varphi_{xx}+(k\varphi_x)_x)w_x^2dxdadt
\\[3pt]&+ s^3 \int_{Q}\varphi_x^2(k\varphi_{xx}+(k\varphi_x)_x)w^2dxdadt
\\[3pt]
&
-2 s^2\int_{Q} \varphi_x\varphi_{xt}w^2dxdadt
 +\frac{s}{2}\int_{Q} \frac{\varphi_{tt}}{k}w^2dxdadt
\\[3pt] &+ s\int_{Q_T}(k\varphi_{xx})_{x}ww_x dxdadt
\\[3pt]&
-\frac{1}{2}\int_0^A\int_0^{1}\Big[w_x^2\Big]_0^Tdxda
+\int_0^T\int_0^A\Big[w_xw_t\Big]_0^{1}dadt
\\[3pt]&
- s \int_0^T\int_0^A\Big[ a \varphi_xw_x^2\Big]_0^{1}dadt\\[3pt]&
- s \int_0^T\int_0^A\!\!\Big[ a\varphi_{xx}ww_x\Big]_0^{1}dadt
+\frac{1}{2}\int_0^A \int_{0}^{1}\!\!\Big[ (s^2 \varphi_x^2 - s \frac{\varphi_t}{a})w^2\Big]_0^Tdxda
\\[3pt]
&
- s \int_0^T\int_0^A\Big[  (s^2 a \varphi_x^3- s\varphi_x\varphi_t)w^2\Big]_0^{1}dadt.
\end{aligned}
\end{equation}
Next, we compute $I_3$ and $I_4$
\begin{equation}\label{!}
\begin{aligned}
I_3=&
- \int_{Q} w_{xa} w_x dxdadt
+ \int_0^T\int_0^A\big[ w_xw_a\big]_0^{1} dadt
\\[3pt]
&
+\int_{Q}\big(s^2\varphi_x^2  - s \frac{\varphi_t+ \varphi_a}{k}\big)w w_adxdadt
\\[3pt]
& = -\frac{1}{2}\int_0^T\int_0^{1} \big[w_x^2\big]_0^Adxdt
+\int_0^T\int_0^A\big[w_xw_a\big]_0^{1}dadt  +\frac{1}{2}\int_0^T\int_0^{1}\big[ \big(s^2 \varphi_x^2 - s \frac{\varphi_t+\varphi_a}{a} \big)w^2\big]_0^Adxdt
\\[3pt]
&
+
\frac{1}{2} \int_{Q} \left(-s^2\varphi_x^2 + s \frac{\varphi_t+ \varphi_a}{k}\right)_aw^2dxdadt
\\[3pt]
&=-\frac{1}{2}\int_0^T\int_0^{1} \big[w_x^2\big]_0^Adxdt
+\int_0^T\int_0^A\big[w_xw_a\big]_0^{1}dadt +\frac{1}{2}\int_0^T\int_0^{1}\big[ \big(s^2 \varphi_x^2 - s \frac{\varphi_t+\varphi_a}{a} \big)w^2\big]_0^Adxdt
\\[3pt]
&
+
\frac{s}{2} \int_{Q} \frac{\varphi_{aa}}{k}w^2dxdadt+
\frac{s}{2} \int_{Q} \frac{\varphi_{ta}}{k}w^2dxdadt
- s^2\int_{Q} \varphi_x\varphi_{xa}w^2dxdadt.
\end{aligned}
\end{equation}
On the other hand
\begin{equation}\label{!2}
\begin{aligned}
I_4&= -s\int_{Q}\frac{\varphi_a ww_t}{k} dxdadt+ 2s^2\int_{Q}\varphi_x\varphi_a w w_x dxdadt+
 s^2\int_{Q}\varphi_a\varphi_{xx}w^2 dxdadt\\
 &=-\frac{s}{2}\int_{Q}\frac{\varphi_a }{k}(w^2)_t dxdadt+ s^2\int_{Q}\varphi_x\varphi_a (w^2)_xdxdadt+
 s^2\int_{Q}\varphi_a\varphi_{xx}w^2 dxdadt\\
 & =\frac{s}{2}\int_{Q}\frac{\varphi_{at} }{k}w^2 dxdadt- s^2\int_{Q}(\varphi_x\varphi_a)_xw^2dxdadt+
 s^2\int_{Q}\varphi_a\varphi_{xx}w^2 dxdadt\\
 &-\frac{s}{2}\int_0^A\int_0^{1} \left[\frac{\varphi_a }{k}w^2\right]_0^T dxda+ s^2 \int_0^T\int_0^A\big[\varphi_x\varphi_aw^2\big]_0^{1}dadt \\
 & =\frac{s}{2}\int_{Q}\frac{\varphi_{at} }{k}w^2 dxdadt-
 s^2\int_{Q}\varphi_x\varphi_{ax}w^2 dxdadt\\
 &-\frac{s}{2}\int_0^A\int_0^{1} \left[\frac{\varphi_a }{k}w^2\right]_0^T dxda+ s^2 \int_0^T\int_0^A\big[\varphi_x\varphi_aw^2\big]_0^{1}dadt.
\end{aligned}
\end{equation}
Adding \eqref{!1} - \eqref{!2}, \eqref{D&BT} follows immediately.
\end{proof}
The next lemma holds.
\begin{Lemma}\label{BT}Assume Hypothesis $\ref{BAss01}$.
The boundary terms in \eqref{D&BT} become
\begin{equation}\label{BT1}
\begin{aligned}
\{B.T.\}=-se^{R}\int_0^T\int_0^A\Theta(t,a)w_x^2(t, a, 1)dadt.
\end{aligned}
\end{equation}
\end{Lemma}
The proof is based on the next result:
\begin{Lemma} \label{rem}
For all $\gamma \ge M$ the map
$ x \mapsto \dfrac{x^\gamma}{k}$ is nondecreasing in
$(0,1]$
 so that $\lim_{x\to 0}\dfrac{x^\gamma}{k}=0 $
for all $\gamma>M$.
\end{Lemma}
\begin{proof}[Proof of Lemma \ref{BT}] Using the definition of $\varphi$ and \cite[Lemma 3.9]{cfr},
the boundary terms of  \:$<L^+_sw,
L^-_sw>_{L^2_{\frac{1}{k}}(Q)}$ become
\[
\begin{aligned}
 \big\{B.T.\big\}\;
&=-se^{R}\int_0^T\int_0^A\Theta(t)w_x^2(t,a,1)dadt
+\int_0^T\int_0^A\Big[w_x w_a\Big]_0^{1}dadt
\\[3pt]&
-\frac{s}{2}\int_0^A \int_{0}^{1}\!\!\Big[ \frac{\varphi_a }{k}w^2\Big]_0^Tdxda
+ s^2 \int_0^T\int_0^A\Big[  \varphi_x\varphi_aw^2\Big]_0^{1}dadt\\
&
 -\frac{1}{2}\int_0^T\int_0^{1} \big[w_x^2\big]_0^Adxdt
 -\frac{s}{2}\int_0^T\int_0^{1}\big[ \frac{\varphi_a}{k}w^2\big]_0^Adxdt\\
  &=\; -se^{R}\int_0^T\int_0^A\Theta(t)w_x^2(t,a,1)dadt
 -\frac{1}{2}\int_0^T\int_0^{1} \big[w_x^2\big]_0^Adxdt
+\int_0^T\int_0^A\Big[w_xw_a\Big]_0^{1}dadt
\\[3pt]
&
- \frac{s}{2}\int_0^A \int_{0}^{1}\!\!\Big[\Theta_a \frac{p(x)-2\|p~\|_{L^{\infty}(0,1)}}{k}w^2\Big]_0^Tdxda- \frac{s}{2}\int_0^T \int_{0}^{1}\!\!\Big[\Theta_a \frac{p(x)-2\|p~\|_{L^{\infty}(0,1)}}{k}w^2\Big]_0^Adxdt\\
&+ s^2  \int_0^T\int_0^A\Big[ \Theta \Theta_a \frac{x}{k}(p(x)-2\|p~\|_{L^{\infty}(0,1)}) e^{Rx^2}w^2\Big]_0^{1}dadt.
\end{aligned}
\]
Since $w \in \mathcal{V}_1$, $w(0, a, x)$, $w(T,a, x)$, $w_x(0,a, x)$,
$w_x(T,a, x)$, $w(t,0, x)$, $w(t, A, x)$ and $\int_0^T\int_0^{1} \big[w_x^2\big]_0^Adxdt$ are well defined; thus, using the boundary conditions and the definition of $w$ itself, we
get
\[
\begin{aligned}
\int_0^T\int_0^{1} \big[w_x^2\big]_0^Adxdt&=\int_0^A \int_{0}^{1}\!\!\Big[\Theta_a \frac{p(x)-2\|p~\|_{L^{\infty}(0,1)}}{k}w^2\Big]_0^Tdxda\\&=\int_0^T \int_{0}^{1}\!\!\Big[\Theta_a \frac{p(x)-2\|p~\|_{L^{\infty}(0,1)}}{k}w^2\Big]_0^Adxdt=0.
\end{aligned}
\]
Moreover, since $w \in \mathcal{V}_1$, we have that
 $w_a(t,a,0)$ and $w_a(t,a,1)$ make sense. Moreover, also $w_x(t,a,0)$ and $w_x(t,a, 1)$
are well defined, since  $w(t,a,\cdot)\in H^2_{\frac{1}{k}}(0,1)$.
Thus $ \int_0^T\int_0^A[ w_xw_a]_{x=0}^{x=1}dadt$ is well defined and actually
equals $0$. Indeed, by the boundary conditions, we find
\[
|w_a(t,a,x)|\le \int_0^x |w_{ax}(t,a,y)| dy \le \sqrt{x}\left(\int_0^x
|w_{ax}(t,a,y)|^2dy\right)^{1/2}\to 0
\]
as $x\to 0$, the integral being finite.
Now, we consider the term
\[
 \int_0^T\int_0^A\Big[ \Theta \Theta_a \frac{x}{k}(p(x)-2\|p~\|_{L^{\infty}(0,1)}) e^{Rx^2}w^2\Big]_0^{1}dadt.
 \]
 Since $w(t,a,1)=0$, 
 \[
  \int_0^T\int_0^A\Big[ \Theta \Theta_a \frac{x}{k}(p(x)-2\|p~\|_{L^{\infty}(0,1)}) e^{Rx^2}w^2\Big](t,a,1)dadt=0.
 \]
Moreover, by H\"older inequality, $\ds w^2(t,a,x) \le x \int_0^x w_x^2(t,a,y)dy$; hence, by Lemma \ref{rem}, one has
\[
\left|\Theta(t,a){\Theta_a}(t,a)\frac{~x}{k}~w^2(t,a,x)\right|
\le\:
\Theta(t,a)|{\Theta_a}(t,a)|\frac{x ^2}{k(x)}
\int_0^{\text{\:\small$x$}}w_x^2(t,a,y) dy \rightarrow 0,
\]
as $x \rightarrow 0^+$. Thus
\[
\begin{aligned}
&2s^2\|p~\|_{L^{\infty}(0,1)}\int_0^T\int_0^A\Theta(t,a){\Theta_a}(t,a)\left[e^{Rx^2}\frac{~x}{k}~w^2\right](t,a,0) dadt 
\\[3pt]& =
\lim_{\epsilon\rightarrow 0}
2s^2\|p~\|_{L^{\infty}(0,1)}\int_0^T\int_0^A\Theta(t,a){\Theta_a}(t, a)\left[e^{Rx^2}\frac{~x}{k}~w^2\right](t,a,\epsilon) dadt=0.
\end{aligned}
\]
Finally, using the fact that the function $\ds x \mapsto \frac{x^{M}}{k}$ is nondecreasing, one has that
\[
\int_0^T\int_0^A\Theta(t,a){\Theta_a}(t,a)\left[e^{Rx^2}\frac{~x}{k}p(x)w^2\right](t,a,0) dadt =0.
\]
Indeed, if $M\le 1$,
\[
\left|\Theta(t,a){\Theta_a}(t,a)e^{Rx^2}\frac{~x}{k}p(x)w^2(t,a,x)\right|\le  \Theta(t,a)|{\Theta_a}(t,a)|e^R\left(\frac{x^2}{k(x)}\right)^2\int_0^x w_x^2(t,a,y)dy \rightarrow 0, \]
as $x \rightarrow 0$.
If $M >1$,
\[
\begin{aligned}
\left|\Theta(t,a){\Theta_a}(t,a)e^{Rx^2}\frac{~x}{k}p(x)w^2(t,a,x)\right|&\le   \Theta(t,a)|{\Theta_a}(t,a)|e^R\frac{x^{M+1}}{k^2(x)}w^2(t,a,x)\int_0^x \frac{1}{y^{M-1}} dy \\
&=\Theta(t,a)|{\Theta_a}(t,a)|e^R\frac{x^3}{k^2(x)}w^2(t,a,x).
\\
&\le \Theta(t,a)|{\Theta_a}(t,a)|e^R\frac{x^4}{k^2(x)}\int_0^x w_x^2(t,a,y)dy\rightarrow 0, 
\end{aligned}
\]
as $x \rightarrow 0$.
Hence the thesis.
\end{proof}

The crucial step is to prove
now the following estimate.

\begin{Lemma}\label{lemma2}Assume Hypothesis $\ref{BAss01}$. There exist two  strictly positive constants $C$ and $s_0$ such
that, for all $s\ge s_0$,  all solutions $w$ of \eqref{1'}
satisfy the following estimate
\[
sC\int_{Q}\Theta w_x^2 dxdadt
+s^3C\int_{Q}\Theta^3 \Big(\frac{x}{k}\Big)^2w^2 dxdadt \le \big\{D.T.\big\} .
\]
\end{Lemma}
\begin{proof}
The distributed terms of
\:$<L^+_sw, L^-_sw>_{L^2_{\frac{1}{k}}(Q)}$, using the definition of $\varphi$, take the form
\begin{equation}\label{02}
\begin{aligned}
\big\{D.T.\big\}\; &=
\quad s \int_{Q}\Theta\Big(2-\frac{xk_x}{k}+ 4Rx^2\Big)e^{Rx^2}w_x^2dxdadt
\\[5pt]
&+s^3\int_{Q}\Theta^3\Big(\frac{x}{k}\Big)^2\Big(2-\frac{xk_x}{k}+ 4Rx^2\Big)e^{3Rx^2}w^2dxdadt
\\[5pt]
&-2 s^2\int_{Q}\Theta{\Theta_t}\Big(\frac{x}{k}\Big)^2e^{2Rx^2}w^2dxdadt
+\frac{s}{2} \int_{Q}\frac{{\Theta_{tt}}}{k}\Big(p- 2\|p\|_{L^{\infty}(0,1)}\Big)w^2dxdadt
\\[5pt]
&+ s \int_Q \Theta\left( e^{Rx^2}\left[1+2Rx- x \frac{k'}{k}\right]\right)_x w w_xdxdadt\\
&+\frac{s}{2} \int_{Q}\frac{{\Theta_{aa}}}{k}\Big(p- 2\|p\|_{L^{\infty}(0,1)}\Big)w^2dxdadt\\
&+s \int_{Q}\frac{{\Theta_{ta}}}{k}\Big(p- 2\|p\|_{L^{\infty}(0,1)}\Big)w^2dxdadt-2s^2 \int_{Q}\Theta \Theta_a \left(\frac{x}{k}\right)^2e^{2Rx^2}w^2dxdadt.
\end{aligned}
\end{equation}
Now, observe that 
 there exists $c>0$ such that 
 \begin{equation}\label{magtheta}
 \begin{aligned}
 &\Theta^{\mu} \le c \Theta ^\nu \mbox{ if } 0<\mu<\nu\\
 &|\Theta  \Theta_t|
\le c\Theta ^3, |\Theta {\Theta_a}|
\le c\Theta ^3, \\
&|{\Theta_{aa}}| \le c\Theta ^{\frac{3}{2}}, |{\Theta_{tt}}| \le c\Theta ^{\frac{3}{2}}
\text{ and  } |\Theta_{ta}|\le c\Theta ^{\frac{3}{2}}.
\end{aligned}
\end{equation}
Hence, proceeding as in the proof of \cite[Lemma 3.8]{cfr} or of \cite[Lemma 4.3]{fm1}, one can deduce
\begin{equation}\label{terminenuovo0}
\begin{aligned}
&s \int_{Q}\Theta\Big(2-\frac{xk_x}{k}+ 4Rx^2\Big)e^{Rx^2}w_x^2dxdadt
\\[5pt]
&+s^3\int_{Q}\Theta^3\Big(\frac{x}{k}\Big)^2\Big(2-\frac{xk_x}{k}+ 4Rx^2\Big)e^{3Rx^2}w^2dxdadt
\\[5pt]
&-2 s^2\int_{Q}\Theta{\Theta_t}\Big(\frac{x}{k}\Big)^2e^{2Rx^2}w^2dxdadt
+\frac{s}{2} \int_{Q}\frac{{\Theta_{tt}}}{k}\Big(p- 2\|p\|_{L^{\infty}(0,1)}\Big)w^2dxdadt
\\[5pt]
&+\frac{s}{2} \int_{Q}\frac{{\Theta_{aa}}}{k}\Big(p- 2\|p\|_{L^{\infty}(0,1)}\Big)w^2dxdadt\\
&+s \int_{Q}\frac{{\Theta_{ta}}}{k}\Big(p- 2\|p\|_{L^{\infty}(0,1)}\Big)w^2dxdadt-2s^2 \int_{Q}\Theta \Theta_a \left(\frac{x}{k}\right)^2e^{2Rx^2}w^2dxdadt\\
&\ge sC\int_{Q}\Theta w_x^2 dxdadt
+s^3C\int_{Q}\Theta^3 \Big(\frac{x}{k}\Big)^2w^2 dxdadt
\\[3pt]
&
-s^2\frac{C}{4}\int_{Q}\Theta^3\left(\frac{x}{k}\right)^2 w^2 dxdadt
- s\frac{C}{4}\ \int_{Q}\frac{~\Theta^{\frac{3}{2}}}{k}w^2  dxdadt,
\end{aligned}
\end{equation}
where $C>0$ denotes some universal  strictly positive constant which may vary
from line to line.

Now, consider the term $\ds \int_Q \Theta\left( e^{Rx^2}\left[1+2Rx- x \frac{k'}{k}\right]\right)_x w w_xdxdadt$.
Setting 
\[\fh:=\ds e^{Rx^2}\left[1+2Rx- x \frac{k'}{k}\right]\]
 and for $\epsilon >0$,  it results
\begin{equation}\label{terminenuovo}
\begin{aligned}
& \left|s \int_Q\Theta \fh' ww_x dxdadt\right| \le \frac{1}{\ve} s \int_Q\Theta |\fh'|^2 w^2dxdadt +\ve s\int_Q\Theta  (w_x)^2 dxdadt\\
& \le \frac{1}{\ve} s c\|\fh'\|^2_{L^\infty(0,1)}\|k\|_{L^\infty(0,1)}\int_Q\Theta^{\frac{3}{2}} \frac{w^2}{k}dxdadt +\ve s \int_Q\Theta (w_x)^2dxdadt.
\end{aligned}
\end{equation}
As in \cite{cfr}, one has, for $\gamma >0$,
\[
\begin{aligned}
 \int_{Q}\frac{~\Theta^{\frac{3}{2}}}{k}w^2 dxdadt
&=
\int_{Q} \left(\frac{1}{\gamma }\Theta^2\left(\frac{x}{k}\right)^2w^2\right)^{\frac{1}{2}}
           \left(\gamma \frac{\Theta}{~x^2}\, w^2\right)^{\frac{1}{2}}dxdadt
\\[3pt]&\le
\frac{1}{\gamma }\int_{Q}\Theta^2\left(\frac{x}{k}\right)^2w^2 dxdadt
+ \gamma \int_{Q}\frac{\Theta}{~x^2}\, w^2 dxdadt.
\end{aligned}\]
By Hardy's inequality one has
\begin{equation}\label{terminenuovo1}
\begin{aligned}
 \int_{Q}\frac{~\Theta^{\frac{3}{2}}}{k}~w^2 dxdadt
& \le
\frac{1}{\gamma }\int_{Q}\Theta^2\left(\frac{x}{k}\right)^2w^2 dxdadt
+\gamma  C\int_{Q}\Theta w_x^2 dxdadt,
\end{aligned}\end{equation}
for a  strictly positive constant $C.$ 

Thus, for $s_0$ large enough and $\gamma $ small enough, by \eqref{terminenuovo0}, \eqref{terminenuovo} and \eqref{terminenuovo1}, the thesis follows.
\end{proof}

As a consequence of Lemmas \ref{BT} and \ref{lemma2}, we have
\begin{Proposition}\label{stima}Assume Hypothesis $\ref{BAss01}$.
There exist two  strictly positive constants $C$ and $s_0$ such
that, for all $s\ge s_0$,  all solutions $w$ of \eqref{1'} in $\mathcal{V}_1$ satisfy
\[
\int_{Q}\!\!\!  s\Theta w_x^2+s^3\Theta^3 \left(\frac{x}{k}\right)^2 w^2 dxdadt
\le
C\left(\int_{Q}f^{2}\text{\small$\frac{~e^{2s\varphi}}{k}$\normalsize}~dxdadt
+ s\int_0^T\int_0^A\Theta(t,a)w_x^2(t,a, 1)dadt\right).
\]
\end{Proposition}

 Recalling the definition of $w$, we have $v= e^{-s\varphi}w$
and $v_{x}=  (w_{x}-s\varphi_{x}w)e^{-s\varphi}$. Thus, Theorem
\ref{Cor1} follows immediately by Proposition \ref{stima} when $\mu \equiv 0$.

\vspace{0.5cm}
Now, we assume that $\mu \not \equiv 0$.

To complete the proof of Theorem \ref{nondegenerediv} we consider the function $\overline{f}=f+\mu v$.
 Hence,  there are  two  strictly positive constants $C$ and $s_0$ such that, for
all $s\geq s_0$, the following inequality holds
\begin{equation} \label{fati1?}
\begin{aligned}
\int_{Q}\left(s \Theta v_x^2
                + s^3\Theta^3\text{\small$\displaystyle\Big(\frac{x}{k}\Big)^2$\normalsize}
                  v^2\right)e^{2s\varphi}dxdadt
&\le
C\int_{Q}\bar f^{2}\text{\small$\frac{~e^{2s\varphi}}{k}$\normalsize}~dxdadt
\\&+ s\,C\int_0^T \int_0^A\Theta(t,a)\Big[ x v_x^2 e^{2s\varphi}\Big](t, a, 1)dadt.
\end{aligned}
\end{equation}
On the other hand, we have
\begin{equation} \label{4'?}
\begin{aligned}
\int_{Q}\mid\overline{f} \mid^{2}\frac{~e^{2s\varphi}}{k} \,dxdadt
\leq 2\Big(\int_{Q}|f|^{2}\frac{~e^{2s\varphi}}{k}\,dxdadt
+\|\mu\|_{L^\infty(Q)}^{2}\int_{Q}|v|^{2}\frac{~e^{2s\varphi}}{k}\,dxdadt\Big).
\end{aligned}
\end{equation}
Now, applying Hardy-Poincar\'{e} inequality to the function $\nu:=e^{s\varphi}v$,
we obtain
 \[
      \begin{aligned}
& \int_{Q}|v|^{2}\frac{~e^{2s\varphi}}{k}\,dxdadt=   \int_{Q}
                    \frac{\nu^2}{k}dxdadt
   = \int_{Q}
                    \frac{x^2}{k} \frac{\nu^2}{x^2}dxdadt\le C\int_{Q}
                    \frac{\nu^2}{x^2}dxdadt \\
                    &
     \le
    C\int_{Q}(e^{s\varphi}v)^2_xdxdadt
   \le
    C\int_{Q} e^{2s\varphi}v_x^2dxdadt
   + Cs^2\int_{Q}\Theta^2 e^{2s\varphi}\left(\frac{x}{k}\right)^2 v^2dxdadt.
  \end{aligned}
  \]
   Using this last inequality in (\ref{4'?}), it follows
      \begin{equation}\label{fati2?}
      \begin{aligned}
      \int_{Q} |\bar{f}|^{2}\text{\small$\frac{~e^{2s\varphi}}{k}$\normalsize}~dxdadt
      &\le
      2\int_{Q} |f|^{2}\text{\small$\frac{~e^{2s\varphi}}{k}$\normalsize}~dxdadt
      +C\int_{Q}e^{2s\varphi} v_x^2 dxdadt
      \\&+ Cs^2\int_{Q} \Theta^2 e^{2s\varphi}\left(\frac{x}{k}\right)^2 v^2dxdadt.
     \end{aligned} \end{equation}
   Substituting in
   (\ref{fati1?}), one can conclude

      \[
      \begin{aligned}
    &  \int_{Q}\left(s \Theta v_x^2 +s^3\Theta^3\Big(\frac{x}{k}\Big)^2 v^2\right)e^{2s\varphi}dxdadt
       \le
       C\Big(\int_{Q} |f|^{2}\text{\small$\frac{~e^{2s\varphi}}{k}$\normalsize}~dxdadt
        \\[3pt]& +\int_{Q}\!\!\!e^{2s\varphi} v_x^2 dxdadt
         + s^2\int_{Q}\!\!\! \Theta^2 e^{2s\varphi}\left(\frac{x}{k}\right)^2 v^2dxdadt
         + s\int_0^T \int_0^A\Theta(t,a)\Big[ x v_x^2 e^{2s\varphi}\Big](t, a, 1)dadt \Big).
        \end{aligned}
      \]
This completes the proof of Theorem \ref{Cor1}.
\paragraph{Carleman inequalities when the degeneracy is in the interior}\label{sec-3-3}

Now, we prove  Carleman inequalities for \eqref{adjoint} when $k$ has an interior degeneracy point.
In particular, on $k$ we assume
\begin{Assumptions}\label{Ass021}
The function $k$ satisfies Hypothesis $\ref{Ass0}$ or Hypothesis
$\ref{Ass01}$. Moreover,
\[
\frac{(x-x_0)k'(x)}{k(x)} \in
W^{1,\infty}(0,1),
\]
and, if $M \ge 1$, there exists a constant $\vartheta \in
(0, M]$ such that the function 
\begin{equation}\label{dainfinito}
\begin{array}{ll}
x \mapsto \dfrac{k(x)}{|x-x_0|^{\vartheta}} &
\begin{cases}
& \mbox{ is nonincreasing on the left of $x=x_0$,}\\
& \mbox{ is nondecreasing on the right of $x=x_0$}.
\end{cases}
\end{array}
\end{equation}
\end{Assumptions}
\vspace{0.5cm}

As before, we introduce the function $\Gamma(t,a,x): =\Theta(t,a)\gamma(x)$,
where $\Theta$ is defined as in \eqref{571} and
\begin{equation}\label{c_11}
\gamma(x) := d_1\left(\int_{x_0}^x \frac{y-x_0}{k(y)}e^{R(y-x_0)^2}dy-
d_2\right),
\end{equation}
with $R>0$, $d_2> \displaystyle
\max\left\{\frac{(1-x_0)^2e^{R(1-x_0)^2}}{(2-K)k(1)},
\frac{x_0^2e^{Rx_0^2}}{(2-K)k(0)}\right\}$ and $d_1>0$. Also in this case we have
\[
-d_1d_2\le\gamma(x)<0 \quad \mbox{ for every }x\in[0,1].
\]
\begin{Theorem}\label{Cor11}
Assume Hypothesis  $\ref{Ass021}$. Then,
there exist two  strictly positive constants $C$ and $s_0$ such that every
solution $v$ of \eqref{adjoint} in
\begin{equation}\label{v1}
\mathcal{V}_2:=L^2\big(Q_{T,A}; H^2_{{\frac{1}{k}, x_0}}(0,1)\big) \cap H^1\big(0, T; H^1(0,A;H^1_{{\frac{1}{k}}}(0,1))\big)
\end{equation}
satisfies
\begin{equation}\label{car}
\begin{aligned}
&\int_{Q} \left(s\Theta (v_x)^2 + s^3 \Theta^3
\left(\frac{x-x_0}{k} \right)^2v^2\right)e^{2s\Gamma}dxdadt\\
&\le C\left(\int_{Q} f^{2}\frac{e^{2s\Gamma}}{k}dxdadt +
sd_1\int_0^T\int_0^A\left[\Theta e^{2s \Gamma}(x-x_0)(v_x)^2
dadt\right]_{x=0}^{x=1}dadt\right)
\end{aligned}
\end{equation}
for all $s \ge s_0$, where $d_{1}$ is the constant of
\eqref{c_11}.
\end{Theorem}
\begin{Remark}
Observe that Theorem \ref{Cor11} is  the same as \cite[Theorem 4.2]{fm1}. However, here we assume that $k$ satisfies \eqref{dainfinito} only if $M \ge 1$, while in \cite{fm1} condition \eqref{dainfinito} is required if $M \ge \ds\frac{1}{2}$. Thus, also in this situation, we improve \cite[Theorem 4.2]{fm1} when $y$ is independent of $a$.
\end{Remark}

\paragraph{Proof of Theorem $\ref{Cor11}$}

The proof of Theorem \ref{Cor11} follows the ideas of the one of
\cite[Theorem 4.2]{fm1} or Theorem \ref{Cor1}. As before, we consider, first of all, the case when $\mu\equiv 0$: for
every $s> 0$ consider the function
\[
w(t,a, x) := e^{s \Gamma (t,a, x)}v(t,a, x),
\]
where $v$ is any solution of \eqref{adjoint} in $\mathcal{V}_2$, so that
also $w\in\mathcal{V}_2$, since $\Gamma<0$.
Moreover, $w$ satisfies \eqref{1'} and Lemma \ref{lemma1} still holds. We underline the fact that 
all integrals and integrations by parts are justified by the
definition of $D(\mathcal A)$ and the choice of $\Gamma$, while
before they were guaranteed by the choice of Dirichlet
conditions at $x=0$ or $x=1$, i.e. where the operator degenerates.
Thus we start with the analogue of Lemma \ref{lemma2} in the
weakly and in the strongly degenerate cases, which now gives the
following estimate:

\begin{Lemma}\label{lemma21}
Assume Hypothesis $\ref{Ass021}$. Then there exists a  strictly positive
constant $s_0$ such that for all $s \ge s_{0}$ the distributed terms
of \eqref{D&BT} satisfy the estimate
\[
\begin{aligned}
&s\int_{Q}(k\Gamma_{xx}+(k\Gamma_x)_x)(w_x)^2 dxdadt + s^3
\int_{Q}(\Gamma_x)^2(k\Gamma_{xx}+(k\Gamma_x)_x)w^2 dxdadt
\\
& -2 s^2\int_{Q} \Gamma_x\Gamma_{xt}w^2 dxdadt
 +\frac{s}{2}\int_{Q} \frac{\Gamma_{tt}}{k}w^2 dxdadt
+s\int_{Q}(k\Gamma_{xx})_{x}ww_x  dxdadt\\
&+\frac{s}{2}\int_{Q} \frac{\Gamma_{aa}}{k}w^2 dxdadt+s\int_{Q} \frac{\Gamma_{ta}}{k}w^2 dxdadt-2s^2 \int_{Q} \Gamma_x\Gamma_{xa}w^2 dxdadt
\\&\ge
Cs\int_{Q} \Theta (w_x)^2 dxdadt +Cs^3
\int_{Q}\Theta^3 \left(\frac{x-x_0}{k} \right)^2 w^2
 dxdadt,\end{aligned}
\]
for a universal  strictly positive constant $C$.
\end{Lemma}
\begin{proof}
Using the definition of $\Gamma$, the distributed terms of
$\displaystyle\int_{Q} \frac{1}{k}L^+_s w L^-_s w dxdadt$
take the form
\[
\text{\{D.T.\}}_1\;\left\{\begin{aligned}
&\frac{s}{2}\int_{Q}\frac{\Theta_{tt}+ \Theta_{aa}} {k}\gamma w^2 dxdadt
-2s^2\int_{Q}\Theta{\Theta_t}(\gamma')^2w^2 dxdadt-2s^2\int_{Q}\Theta{\Theta_a}(\gamma')^2w^2 dxdadt\\
&+ s \int_{Q}\Theta (2k\gamma''+ k'
\gamma')(w_x)^2 dxdadt+ s^3\int_{Q}\Theta^3(2k\gamma''+ k' \gamma')(\gamma')^2w^2
dxdadt\\
&
+s\int_{Q} \Theta(k\gamma'')' ww_x dxdadt.
\end{aligned}\right.
\]
Because of the choice of $\gamma(x)$, one has \[\displaystyle 2k(x)
\gamma''(x) + k'(x) \gamma' (x)= d_1e^{R(x-x_0)^2}\frac{2k(x)-k'(x)
(x-x_0)+4R(x-x_0)^2k(x)}{k(x)}.
\]
As in \cite{fm1}, by Hypothesis \ref{Ass0} or \ref{Ass01}, we immediately find
\[
  2-\frac{(x-x_0)k'}{k}+ 4R(x-x_0)^2\ge 2-M \quad \text{a.e. } \; x\in[0,1],
\]
for every $R>0$.
Thus, using the fact that $e^{R(x-x_0)^2}$ is bounded and bounded away from 0 in
$[0,1]$, the distributed terms satisfy the estimate
\begin{equation}\label{aaaaa}
\begin{aligned}
\{D.T.\}_1 &\ge\frac{s}{2}\int_{Q}\frac{\Theta_{tt}+ \Theta_{aa}} {k}\gamma w^2 dxdadt -s^2C\int_{Q}|\Theta\Theta_t|\left(\frac{x-x_0}{k}\right)^2w^2
dxdadt\\
&-s^2C\int_{Q}|\Theta\Theta_a|\left(\frac{x-x_0}{k}\right)^2w^2
dxdadt\\
&+ s C\int_{Q}\Theta (w_x)^2 dxdadt+ s^3C\int_{Q}\Theta^3\left(\frac{x-x_0}{k} \right)^2w^2
dxdadt\\
&+s\int_{Q} \Theta(k\gamma'')' ww_x dxdadt.
\end{aligned}
\end{equation}
By \eqref{magtheta}, we conclude that, for $s$ large enough,
\[
\begin{aligned}
s^2C\int_{Q}(|\Theta\Theta_t|+|\Theta \Theta_a|)\left(\frac{x-x_0}{k}
\right)^2 w^2 dxdadt&\le cC s^2
\int_{Q}\Theta^3\left(\frac{x-x_0}{k} \right)^2w^2 dxdadt\\
&\le \frac{C^3}{8}s^3\int_{Q}\Theta^3 \left(\frac{x-x_0}{k}
\right)^2w^2 dxdadt.
\end{aligned}
\]
Again as in \cite{fm1}, by \eqref{magtheta} we get
\begin{equation}\label{quasfin}
\begin{aligned}
\left|  \frac{s}{2}\int_{Q} \frac{\Theta_{tt} + \Theta_{aa}}{k}\gamma
w^2dxdadt \right| 
&\leq s\frac{d_1d_2}{2}c\int_Q\Theta ^{3/2} \frac{w^2}{k} dxdadt
\\&
\le \frac{C}{4}s\int_{Q} \Theta  (w_x)^2
dxdadt
\\&+ \frac{C^3}{4}s^3\int_{Q}\Theta^3
\left(\frac{x-x_0}{k}\right)^2w^2 dxdadt.
\end{aligned}
\end{equation}
Now, we consider the last term in \eqref{aaaaa}, i.e. $s\int_{Q} \Theta(k\gamma'')' ww_x dxdadt.$
By Hypothesis
\ref{Ass021} and using the definition of $\gamma$, as in \cite{fm1}, we get
\[
\begin{aligned}
& \left|s\int_{Q}\Theta (k \gamma '')'  ww_x dxdadt\right| \le   \frac{C}{4}s\int_{Q} \Theta  (w_x)^2 dxdadt +   s^3
\frac{C^3}{8} \int_{Q}\Theta^3
\left(\frac{x-x_0}{k}\right)^2w^2 dxdadt.
\end{aligned}
\]
Summing up, we obtain
\[
\begin{aligned}
\{D.T.\}_1&\ge -\frac{C}{4}s\int_{Q} \Theta  (w_x)^2 dxdadt -
\frac{C^3}{4}s^3\int_{Q}\Theta^3
\left(\frac{x-x_0}{k}\right)^2w^2 dxdadt \\
& -\frac{C^3}{8}s^3\int_{Q}\Theta^3 \left(\frac{x-x_0}{k}
\right)^2w^2 dxdadt
\\&
+ s C\int_{Q}\Theta (w_x)^2 dxdadt+ s^3C\int_{Q}\Theta^3\left(\frac{x-x_0}{k} \right)^2w^2 dxdadt\\&
-\frac{C}{4}s\int_{Q} \Theta  (w_x)^2 dxdadt -
\frac{C^3}{8}s^3\int_{Q}\Theta^3(w_x)^2 dxdadt
\\&
= \frac{C}{2}s\int_{Q} \Theta (w_x)^2 dxdadt +\frac{C^3}{2}s^3
\int_{Q}\Theta^3 \left(\frac{x-x_0}{k} \right)^2 w^2 dxdadt.
\end{aligned}
\]
\end{proof}

As for the boundary terms, similarly to Lemma \ref{BT}, we
have the following result, whose proof parallels the one of Lemma
\ref{BT} and is thus omitted (see also \cite[Lemma 4.4]{fm1}).
\begin{Lemma}\label{lemma41}
Assume Hypothesis $\ref{Ass021}$. Then the boundary terms in
\eqref{D&BT} reduce to
\[-sd_1\int_0^T\int_0^A\Theta(t)a\Big[(x-x_0)e^{R(x-x_0)^2}(w_x)^2\Big]_{x=0}^{x=1}dadt.
\]
\end{Lemma}

By Lemmas \ref{lemma21} and \ref{lemma41}, there exist $C>0$ and
$s_0>0$ such that all solutions $w$ of \eqref{1'} satisfy, for
all $s \ge s_0$,
\begin{equation}\label{D&BT11}
\begin{aligned}
\int_{Q} \frac{1}{k}L^+_s w L^-_s w dxdadt &\ge
Cs\int_{Q} \Theta (w_x)^2 dxdadt\\&+ Cs^3
\int_{Q}\Theta^3 \left(\frac{x-x_0}{k}\right)^2 w^2 dxdadt
\\&-sd_1\int_0^T\int_0^A\Theta(t,a)\Big[(x-x_0)e^{R(x-x_0)^2}(w_x)^2\Big]_{x=0}^{x=1}dadt.
\end{aligned}
\end{equation}
Thus, for all $s \ge s_0$, we obtain the next
Carleman inequality for $w$:
\[
\begin{aligned}
& s\int_{Q} \Theta (w_x)^2dxdadt  + s^3
\int_{Q}\Theta^3 \left(\frac{x-x_0}{k}\right)^2w^2
dxdadt\\
&\le C\left(\int_{Q} f^2 \frac{e^{2s\Gamma}}{k}dxdadt+
sd_1\int_0^{T}\int_0^A \left[\Theta
(x-x_0)e^{R(x-x_0)^2}(w_x)^2\right]_{x=0}^{x=1}dadt \right).
\end{aligned}
\]

Theorem \ref{Cor11} follows recalling the definition of $w$.

\vspace{0.5cm}
If $\mu \not \equiv 0$, we can 
proceed as in the proof of Theorem  \ref{Cor1}, obtaining the thesis.

\section{Observability and controllability  of linear equations}\label{osservabilita}
In this section we will prove, as a
consequence of the Carleman estimates established in Section 3, 
observability inequalities  for the associated  adjoint problem of \eqref{1}. To this aim, we assume that the control set $\omega$ is such that
\begin{equation}\label{omega}
\omega=  (\alpha, \rho)  \subset\subset  (0,1),
\end{equation}
if $k$ degenerates at the boundary of $(0,1)$. When $k$ degenerates at $x_0\in (0,1),$ $\omega$ is such that 
\begin{equation}\label{omega0}
x_0 \in \omega =  (\alpha, \rho)  \subset\subset  (0,1),
\end{equation}
or
 \begin{equation}\label{omega_new}
 \omega = \omega_1 \cup
\omega_2,
\end{equation}
 where 
\begin{equation}\label{omega2}
\omega_i=(\lambda_i,\rho_i) \subset (0,1), \, i=1,2, \mbox{ and
$\rho_1 < x_0< \lambda_2$}.
\end{equation}

\begin{Remark}\label{beta1}
Observe that, if \eqref{omega0} holds, we can find two subintervals
$\omega_1=(\lambda_1,\rho_1)\subset \subset(\alpha, x_0),
\omega_2=(\lambda_2,\rho_2) \subset\subset (x_0,\rho)$. 
\end{Remark}

Moreover, $k$ and $\beta$ satisfy the following assumptions:
\begin{Assumptions}\label{iposezione4}
The function $k$ is s.t. Hypothesis \ref{BAss01}, \ref{BAss02} or \ref{Ass021} is satisfied. Moreover, if Hypothesis \ref{Ass0} holds, 
there exist
two functions $\fg \in L^\infty_{\rm loc}([-\rho_1,1]\setminus \{x_0\})$, $\fh \in W^{1,\infty}_{\rm loc}([-\rho_1,1]\setminus \{x_0\}, L^\infty(0,1))$ and
two strictly positive constants $\fg_0$, $\fh_0$ such that $\fg(x) \ge \fg_0$ 
\begin{equation}\label{aggiuntivastrana}
\frac{\tilde k'(x)}{2\sqrt{\tilde k(x)}}\left(\int_x^B\fg(t) dt + \fh_0 \right)+ \sqrt{\tilde k(x)}\fg(x) =\fh(x,B)\quad \text{for a.e.} \; x \in [-\rho_1,1], B \in [0,1]
\end{equation}
with $x<B<x_0$ or $x_0<x<B$, where
\begin{equation}\label{tildek}
\tilde k(x):= \begin{cases}k(x), & x \in [0,1],\\
k(-x), & x \in [-1,0].
\end{cases}
\end{equation} 
\end{Assumptions}
Observe that \eqref{aggiuntivastrana} implies the fact that
$
\ds \frac{k'}{\sqrt k} \in L^\infty_{\text{loc}} ([0,1] \setminus\{x_0\}).
$
\begin{Assumptions}\label{conditionbeta} Assume $T<A$ and suppose that there exists  
$\bar a \le T$
 such that
\begin{equation}\label{conditionbeta1}
\beta(a, x)=0 \;  \text{for all $(a, x) \in [0, \bar a]\times [0,1]$}.
\end{equation}
\end{Assumptions} 

Observe that Hypothesis \ref{conditionbeta} is the biological meaningful one. Indeed, $\bar a$ is the minimal age in which the female of the population become fertile, thus it is natural that before $\bar a$ there are no newborns. Obviously,  if $T<A$ and $T=\bar a$, then $y(t,0,x)=\int_T^A \beta(a,x) y(t,a,x) da$. In this case, if $(t,a) \in (0,T) \times (0,T)$, only the mortality rate acts on the equation; hence it is natural to expect that the population is $0$ at $T$.  However, we will prove the observability inequalities  also in this case,  since they are independently interesting. Finally, we underline that, since $T$ is strictly less than $A$, we are able to control the population also in small times, thus complementing \cite{aem}.

\vspace{0,4cm}

Under the previous hypotheses, the following observability inequality holds:

\begin{Proposition}
\label{obser.}
Suppose that Hypotheses $\ref{iposezione4}$ and $\ref{conditionbeta}$ hold and assume that $\omega$ satisfies \eqref{omega}, \eqref{omega0} or \eqref{omega_new}. Then,  for every $\delta \in (T,A)$,
there exists a  strictly positive constant $C= C(\delta)$  such that   every
solution $v\in \mathcal U$ of \begin{equation}\label{h=0}
\begin{cases}
\ds \frac{\partial v}{\partial t} + \frac{\partial v}{\partial a}
+k(x)v_{xx}-\mu(t, a, x)v +\beta(a,x)v(t,0,x)=0,& (t,x,a) \in  Q,
\\[5pt]
v(t,a,0)=v(t,a,1) =0, &(t,a) \in Q_{T,A},\\
  v(T,a,x) = v_T(a,x) \in L^2(Q_{A,1}), &(a,x) \in Q_{A,1} \\
  v(t,A,x)=0, & (t,x) \in Q_{T,1},
\end{cases}
\end{equation}
satisfies
\begin{equation}\label{OI}
 \int_0^A\int_0^1  \frac{1}{k}v^2(T-\bar a,a,x) dxda \le 
 C\left( \int_0^\delta \int_0^1 \frac{v_T^2(a,x)}{k}dxda+ \int_0^T \int_0^A\int_ \omega \frac{v^2}{k} dx dadt\right).
\end{equation}
Here $v_T(a,x)$ is such that $v_T(A,x)=0$ in $(0,1)$.
\end{Proposition}
\begin{Remark}
\begin{enumerate}
\item
If $T= \bar a$, the observability inequality given in the previous proposition is the corresponding of \cite[Proposition 3.1]{aem}, where the authors proved it for the divergence case under different assumptions and with $ T \ge A$.
\item
Moreover, observe that in \eqref{OI}  the presence   of the integral $\ds \int_0^\delta \int_0^1 \frac{v_T^2(a,x)}{k}dxda$ is related to the presence of the term $\beta(a,x) v(t,0,x)$ in the equation of \eqref{h=0}. In fact,  estimating such a term using the method of characteristic lines, we obtain the previous integral. Obviously, if $v_T(a,x)=0$ a.e. in $(0,\delta) \times (0,1)$, we obtain the classical observability inequality.
\end{enumerate}
\end{Remark}

\vspace{0.3cm}

Before proving Proposition \ref{obser.} we will give some
results that will be very helpful. As a first step we
introduce the following class of functions
\[
{\cal W}:=\Big\{ v\;\text{solution of \eqref{h=0}}\;\big|\;v_T \in
D(\mathcal A^2)\Big\}, \] where \[D(\mathcal A^2) = \Big\{u \in
D(\mathcal A) \;\big|\; \mathcal A u \in
D(\mathcal A)\;\Big\}.
\]

Observe that $D({\cal A}^2)$ is densely
defined in $D({\cal A})$ (see, for example, \cite[Lemma 7.2]{b}) and
hence in $L^2_{\frac{1}{k}}(Q_{A,1})$. 
Obviously,
\[\begin{aligned}
{\cal W}= C^1\big([0,T]\:;D(\mathcal A)\big)&\subset \mathcal{V}:=L^2\big(Q_{T,A}; \cH^2_{{\frac{1}{k}}}(0,1)\big) \cap H^1\big(0, T; H^1(0,A;H^1_{{\frac{1}{k}}}(0,1))\big)\big)
 \subset \cal{U}.
\end{aligned}\]

\begin{Proposition}[Caccioppoli's inequality]\label{caccio} Assume Hypothesis $\ref{BAss01}$ or $\ref{BAss02}$.
Let $\omega'$ and $\omega$ two open subintervals of $(0,1)$ such
that $\omega ' \subset\subset \omega\subset\subset (0,1)$. Let $\psi(t,x):=\Theta(t,a)\Psi(x)$, where $\Theta$ is defined in \eqref{571}
and $\Psi\in C^1(0,1)$  is a strictly negative function.
Then, there exist two  strictly  positive constants $C$ and $s_0$ such that, for all $s \ge s_0$,
\begin{equation}\label{caccioeq}
\begin{aligned}
\int_{0}^T\int_0^A \int _{\omega'}   v_x^2e^{2s\psi } dxdadt
\ &\leq \ C\left( \int_{0}^T\int_0^A \int _{\omega}   v^2  dxdadt + \int_Q f^2 e^{2s\psi } dxdadt\right)\\
& \leq \ C\left( \int_{0}^T\int_0^A \int _{\omega}  \frac{ v^2 }{k} dxdadt + \int_Q f^2 \frac{e^{2s\psi }}{k} dxdadt\right),
\end{aligned}
\end{equation}
for every
solution $v$ of \eqref{adjoint}.
\end{Proposition}

\begin{proof}
Let us consider a smooth function $\xi: [0,1] \to \Bbb R$ such that
  $$\begin{cases}
    0 \leq \xi (x)  \leq 1, &  \text{ for all } x \in [0,1], \\
    \xi (x) = 1 ,  &   x \in \omega', \\
    \xi (x)=0, &     x \in (0,1)\setminus\omega.
    \end{cases}$$
Then, integrating by parts one has
\begin{equation}\nonumber
\begin{aligned}
0
=&
\int_0^T\!\!\!\frac{d}{dt}\left(\int_0^A\int_0^1(\xi e^{s\psi})^2v^2dxda\right)dt
\\[3pt]
&=
\int_Q 2s\psi_t(\xi e^{s\psi})^2v^2 + 2(\xi e^{s\psi})^2v(-v_a-kv_{xx}+ \mu v+f) \:dxdadt
\\[3pt]
&=\; 2s\int_Q \psi_t(\xi e^{s\psi})^2v^2dxdadt+ 2s\int_Q \psi_a(\xi e^{s\psi})^2v^2dxdadt+ 2\int_Q
\left(\xi^2e^{2s\psi}k\right)_xvv_xdxdadt \\[3pt]
&+ 2\int_Q
(\xi^2e^{2s\psi}k)v_x^2dxdadt+2\int_Q
\xi^2e^{2s\psi}\mu v^2dxdadt +2\int_Q
\xi^2e^{2s\psi}fvdxdadt.
\end{aligned}
\end{equation}
Hence, using Young's inequality
\begin{equation}\nonumber
\begin{aligned}
2\int_Q (\xi^2e^{2s\psi}k)v_x^2dxdadt
&=
-2s\int_Q \psi_t\left(\xi e^{s\psi}\right)^2v^2dxdadt- 2s\int_Q \psi_a(\xi e^{s\psi})^2v^2dxdadt
\\[3pt]&-2\int_Q \left(\xi^2e^{2s\psi}k\right)_x\frac{\xi e^{s\psi}\sqrt{k}}{\xi e^{s\psi}\sqrt{k}}\:vv_x\:dxdadt-2\int_Q
\xi^2e^{2s\psi}\mu v^2dxdadt\\
&-2\int_Q
\xi^2e^{2s\psi}fvdxdadt
\\[3pt]
&\le -2s\int_Q \psi_t(\xi e^{s\psi})^2v^2dxdadt - 2s\int_Q \psi_a(\xi e^{s\psi})^2v^2dxdadt\\
&+ 4\int_Q \left(\xi
e^{s\psi}\sqrt{k}\:\right)_x^2v^2dxdadt+\int_Q
(\xi^2e^{2s\psi}k)v_x^2dxdadt
\\&+(2\|\mu\|_{L^\infty(Q)}+1)\int_Q
\xi^2v^2dxdadt +\int_Q
\xi^2e^{2s\psi}f^2dxdadt.
\end{aligned}
\end{equation}
Thus,
\begin{equation}\nonumber
\begin{aligned}
&\inf_{\omega'}\{k\}\int_0^T\int_0^A\int_{\omega'}e^{2s\psi}v_x^2dxdadt \\[3pt]&\le
\sup_{\omega\times(0,T)}\Big\{\left| 4\left(\xi
e^{s\psi}\sqrt{k}\:\right)_x^2-2s(\psi_t+\psi_a)(\xi
e^{s\psi})^2\right|\Big\} \int_0^T\int_0^A\int_{\omega}v^2dxdadt+\int_Q
f^2e^{2s\psi } dxdadt.
\end{aligned}
\end{equation}
\end{proof}
\begin{Proposition}[Caccioppoli's inequality]\label{caccio1} Assume Hypothesis $\ref{Ass021}$ and suppose that $\ds \frac{k'}{\sqrt k} \in L^\infty_{\text{loc}} ([0,1] \setminus\{x_0\})$ if Hypothesis $\ref{Ass0}$ holds.
Let $\omega'$ and $\omega$ two open subintervals of $(0,1)$ such
that $\omega ' \subset\subset \omega\subset\subset (0,1)$  and $x_0 \not
\in \overline{\omega}$. Let $\psi(t,x):=\Theta(t,a)\Psi(x)$, where $\Theta$ is defined in \eqref{571}
and $\Psi\in C^1(0,1)$  is a strictly negative function.
Then the thesis of Proposition $\ref{caccio}$ holds.
\end{Proposition}

The proof of the previous result follows the one of Proposition \ref{caccio}. We underline only that,
in this case, $\left(\xi
e^{s\psi}\sqrt{k}\right)_x$ can be estimated by
\[
C\left(e^{2s\psi}+s^2(\psi_x)^2e^{2s\psi}+ e^{2s\psi}\frac{(k')^2}{k}\right)
\]
and
$\displaystyle\frac{(k')^2}{k}$ exists and is bounded in $\omega$ thanks to the assumptions (recall that $x_0\not \in \bar \omega$).
\vspace{0.3cm}

With the aid of Theorems \ref{Cor1}, \ref{Cor1'}, \ref{Cor11} and Propositions \ref{caccio}, \ref{caccio1}, we can now show $\omega-$local Carleman estimates  for \eqref{adjoint}.

\begin{Theorem}\label{Cor2} Assume Hypothesis $\ref{BAss01}$ and suppose that $\omega$ satisfies \eqref{omega}. Then,
there exist two  strictly positive constants $C$ and $s_0$ such that every
solution $v$ of \eqref{adjoint} in
$
\mathcal {V}_1
$ 
satisfies, for all $s \ge s_0$,
\[
\begin{aligned}
\int_{Q}\left(s \Theta v_x^2
                + s^3\Theta^3\text{\small$\displaystyle\Big(\frac{x}{k}\Big)^2$\normalsize}
                  v^2\right)e^{2s\varphi}dxdadt
&\le
C\left(\int_{Q}f^{2}\text{\small$\frac{~e^{2s\Phi}}{k}$\normalsize}~dxdadt+ \int_0^T \int_0^A\int_ \omega \frac{v^2}{k} dx dadt\right).
\end{aligned}\]
\end{Theorem}
\begin{proof}
Let us consider a smooth function $\xi: [0,1] \to \Bbb R$ such that
  $$\begin{cases}
    0 \leq \xi (x)  \leq 1, &  \text{ for all } x \in [0,1], \\
    \xi (x) = 1 ,  &   x \in [0, (2\alpha +\rho)/3], \\
    \xi (x)=0, &     x \in [(\alpha +2\rho)/3,1].
    \end{cases}$$
    We define $w(t,a,x):= \xi(x)v(t,a, x)$ where $v\in\cal{V}_1$ satisfies \eqref{adjoint}.
    Then $w$   satisfies
    \[
    \begin{cases}
      w_t +w_a+k  w_{xx}- \mu w= \xi f + k (\xi_{xx} v + 2 \xi_xv_x)=:h,&
      (t,a, x) \in Q,
      \\[5pt]
  w(t,a,0)=  w(t,a,1)=0, & (t,a) \in Q_{T,A}.
    \end{cases}
    \]
    Thus, applying Theorem \ref{Cor1} and Proposition \ref{caccio},
    \begin{equation}\label{add1}
    \begin{aligned}
&\int_0^T\int_0^A \int_0^{\frac{2\alpha+ \rho}{3}}\left(s \Theta v_x^2
                + s^3\Theta^3\text{\small$\displaystyle\Big(\frac{x}{k}\Big)^2$\normalsize}
                  v^2\right)e^{2s\varphi}dxdadt \\
                  &=\int_0^T\int_0^A \int_0^{\frac{2\alpha+ \rho}{3}}\left(s \Theta w_x^2
                + s^3\Theta^3\text{\small$\displaystyle\Big(\frac{x}{k}\Big)^2$\normalsize}
                  w^2\right)e^{2s\varphi}dxdadt\\
&\le
                  \int_{Q}\left(s \Theta w_x^2
                + s^3\Theta^3\text{\small$\displaystyle\Big(\frac{x}{k}\Big)^2$\normalsize}
                  w^2\right)e^{2s\varphi}dxdadt
\le
C\int_{Q}h^{2}\text{\small$\frac{~e^{2s\varphi}}{k}$\normalsize}~dxdadt
\\&\le C\left( \int_{Q} f^2 \frac{e^{2s\varphi} }{k} dxdadt + \int_{0}^T\int_0^A \int _{\omega'}  \frac{ v^2 }{k} dxdadt + \int_{0}^T\int_0^A \int _{\omega'}  v_x^2 e^{2s\varphi} dxdadt \right)
\\&
\le C \left( \int_{Q} f^2 \frac{e^{2s\varphi}}{k} dxdadt+ \int_{0}^T\int_0^A \int _{\omega}  \frac{ v^2 }{k} dxdadt \right),
\end{aligned}
\end{equation}
where $\omega':= \ds\left(\frac{2\alpha +\rho}{3}, \frac{\alpha +2\rho}{3}\right)$.

Now, consider $z= \eta v$, where $\eta = 1-\xi$ and take $\bar \alpha \in (0, \alpha)$. Then $z$   satisfies
    \begin{equation}\label{problemz}
    \begin{cases}
      z_t +z_a+k z_{xx}- \mu z= \eta f + k (\eta_{xx} v+ 2 \eta_xv_x)=:h,&
      (t,a, x) \in Q_{T,A}\times (\bar \alpha, 1)=: \bar Q,
      \\[5pt]
  z(t,a,\bar\alpha)=  z(t,a,1)=0, & (t,a) \in Q_{T,A}.
    \end{cases}
    \end{equation}
    Clearly the equation satisfied by $z$ is not degenerate, thus applying Theorem \ref{nondegenere} and Proposition \ref{caccio}, one has
    \[
    \begin{aligned}
&\int_{\bar Q}(s^{3}\phi^{3}z^{2}+s\phi z_{x}^{2})e^{2s\Phi} dxdadt \leq C \int_{\bar Q}h^{2}e^{2s\Phi}dxdadt
\\
 &\le C\left( \int_{\bar Q}f^2 e^{2s\Phi}dxdadt + \int_0^T\int_0^A \int_{\omega'} (v^2 + v_x^2)e^{2s\Phi}dxdadt\right)\\
 &\le C \left( \int_Q\frac{f^{2}}{k}e^{2s\Phi}dxdadt+  \int_0^T\int_0^A \int_{\omega} \frac{v^2}{k}dxdadt\right).
\end{aligned}
    \]
    Hence
    \[ 
      \begin{aligned}
  &  \int_0^T\int_0^A \int_{\frac{\alpha+ 2\rho}{3}}^1   (s^{3}\phi^{3}v^{2}+s\phi v_{x}^{2})e^{2s\Phi} dxdadt=
 \int_0^T\int_0^A \int_{\frac{\alpha+ 2\rho}{3}}^1   (s^{3}\phi^{3}z^{2}+s\phi z_{x}^{2})e^{2s\Phi} dxdadt\\
 &\le C \left( \int_Q\frac{f^{2}}{k}e^{2s\Phi}dxdadt+  \int_0^T\int_0^A \int_{\omega} \frac{v^2}{k}dxdadt\right),
 \end{aligned}
    \]
    for a  strictly positive constant $C$.
Proceeding, for example, as in \cite{fm1}  one can prove the existence of $\varsigma>0$, such that,  for
all $(t,a,x)\in [0,T]\times[0,A]\times[\bar \alpha,1]$, we have
\begin{equation}\label{stimaphi}
e^{2s\varphi}\leq\varsigma e^{2s\Phi},
\left(\frac{x}{k(x)}\right)^2e^{2s\varphi}\leq\varsigma e^{2s\Phi}.
\end{equation}
Thus, for a  strictly  positive constant $C$,
   \begin{equation}\label{add2}
    \begin{aligned}
&\int_0^T\int_0^A \int_{\frac{\alpha+ 2\rho}{3}}^1 \left(s \Theta v_x^2
                + s^3\Theta^3\left(\frac{x}{k}\right)^2
                  v^2\right)e^{2s\varphi}dxdadt \\
                  &\le C\left( \int_0^T\int_0^A \int_{\frac{\alpha+ 2\rho}{3}}^1   (s^{3}\phi^{3}v^{2}+s\phi v_{x}^{2})e^{2s\Phi} dxdadt\right)\\
                 &
\le C \left( \int_{Q} f^2 \frac{e^{2s\Phi}}{k} dxdadt+ \int_{0}^T\int_0^A \int _{\omega}  \frac{ v^2 }{k} dxdadt \right).
\end{aligned}
\end{equation}
Now, consider $\tilde \alpha \in (\alpha, (2\alpha +\rho)/3)$, $\tilde \rho \in ((\alpha +2\rho)/3, \rho)$ and a smooth function $\tau: [0,1] \to \Bbb R$ such that
  $$\begin{cases}
    0 \leq \tau (x)  \leq 1, &  \text{ for all } x \in [0,1], \\
    \tau(x) = 1 ,  &   x \in [ (2\alpha +\rho)/3, (\alpha +2\rho)/3], \\
    \tau (x)=0, &     x \in [0, \tilde \alpha] \cup [\tilde \rho,1],
    \end{cases}$$
    and define $\zeta(t,a,x) := \tau(x)v(t,a,x)$. Clearly, $\zeta$ satisfies \eqref{problemz} with $h:= \tau f + k (\tau_{xx} v+ 2 \tau_xv_x)$. Observe that in this case $\tau_ x, \tau_{xx} \not \equiv 0$ in $\bar\omega:= \ds\left(\tilde\alpha, \frac{2\alpha +\rho}{3}\right) \cup \left( \frac{\alpha +2\rho}{3}, \tilde\rho\right)$. As before, by Theorem \ref{nondegenere}, Proposition \ref{caccio} and \eqref{stimaphi}, we have
      \begin{equation}\label{add3}
    \begin{aligned}
&\int_0^T\int_0^A \int_{\frac{2\alpha+ \rho}{3}}^{\frac{\alpha+ 2\rho}{3}}\left(s \Theta v_x^2
                + s^3\Theta^3\left(\frac{x}{k}\right)^2
                  v^2\right)e^{2s\varphi}dxdadt \\
                  &\le C\left( \int_0^T\int_0^A \int_{\frac{2\alpha+ \rho}{3}}^{\frac{\alpha+ 2\rho}{3}} (s^{3}\phi^{3}v^{2}+s\phi v_{x}^{2})e^{2s\Phi} dxdadt\right)\\
               &   = C\left(\int_0^T\int_0^A \int_{\frac{2\alpha+ \rho}{3}}^{\frac{\alpha+ 2\rho}{3}}   (s^{3}\phi^{3}\zeta^{2}+s\phi \zeta_{x}^{2})e^{2s\Phi} dxdadt\right)
                  \\
                 &
\le C \left( \int_{Q} f^2 \frac{e^{2s\Phi}}{k} dxdadt+ \int_{0}^T\int_0^A \int _{\omega}  \frac{ v^2 }{k} dxdadt \right).
\end{aligned}
\end{equation}
Adding \eqref{add1}, \eqref{add2} and \eqref{add3}, the thesis follows.
\end{proof}
Proceeding as before one can prove

\begin{Theorem}\label{Cor12} Assume Hypothesis $\ref{BAss02}$ and suppose that $\omega$ satisfies \eqref{omega}. Then,
there exist two  strictly positive constants $C$ and $s_0$ such that every
solution $v$ of \eqref{adjoint} in
$
\mathcal {V}_1
$
satisfies, for all $s \ge s_0$,
\begin{equation}\nonumber
\begin{aligned}
&\int_{Q}\left(s \Theta v_x^2
                + s^3\Theta^3\left(\frac{1-x}{k}\right)^2
                  v^2\right)e^{2s\varphi}dxdadt
\\&\le
C\left(\int_{Q}f^{2}\frac{~e^{2s\Phi}}{k}dxdadt+ \int_0^T \int_0^A\int_ \omega \frac{v^2}{k} dx dadt\right).
\end{aligned}\end{equation}
\end{Theorem}

The $\omega-$local Carleman estimates given in Theorems \ref{Cor2} and \ref{Cor12} hold also if $k$ degenerates in the interior of the space domain:
\begin{Theorem}\label{Cor13}Assume  Hypothesis $\ref{Ass021}$ and $\eqref{aggiuntivastrana}$ if Hypothesis $\ref{Ass0}$ holds. Suppose that $\omega$ satisfies \eqref{omega0} or \eqref{omega_new}. Then,
there exist two  strictly  positive constants $C$ and $s_0$ such that every
solution $v$ of \eqref{adjoint} in
$
\mathcal {V}_2
$ 
satisfies, for all $s \ge s_0$,
\begin{equation}\nonumber
\begin{aligned}
&\int_{Q} \left(s\Theta (v_x)^2 + s^3 \Theta^3
\left(\frac{x-x_0}{k} \right)^2v^2\right)e^{2s\Gamma}dxdadt
\\&\le
C\left(\int_{Q}\frac{f^2}{k}dxdadt+ \int_0^T \int_0^A\int_ \omega \frac{v^2}{k} dx dadt\right).
\end{aligned}\end{equation}
\end{Theorem}
\begin{proof} First of all assume that  $\omega=(\alpha, \rho) \subset (0,1)$ is such that $x_0 \in \omega$ and take $\omega_i$, $i=1,2$, as in Remark \ref{beta1}. Now, fix $\bar \lambda_i, \bar \rho_i  \in \omega_i=(\lambda_i, \rho_i)$, $i=1,2$, such that $\bar \lambda_ i <  \bar \rho_i$ and  consider a smooth
function $\xi:[0,1]\to[0,1]$ such that
\[
\xi(x)=\begin{cases} 0&x\in [0,\bar\lambda_1],\\
1 & x\in[\tilde \lambda_1,\tilde \lambda_2],\\
0&x\in [\bar\rho_2,1],
\end{cases}
\]
where $\tilde \lambda_i=(\bar \lambda_i+\bar \rho_i)/2$, $i=1,2$. 
Then, define $w:= \xi v$, where $v$ is any fixed solution of
\eqref{adjoint}. Hence, neglecting the final--time datum (of no
interest in this context), $w$ satisfies
\[
\begin{cases}
w_t + w_a + k  w_{xx}- \mu w =\xi f + k (\xi _{xx} v + 2\xi _x  v_x )=:F,&
(t,x) \in Q, \\
w(t,a, 0)= w(t,a, 1)=0, & t \in Q_{T,A}.
\end{cases}
\]
Applying Theorem \ref{Cor11} and using the fact that $w\equiv0$ in a
neighborhood of $x=0$  and $x=1$, we have
\begin{equation}\label{car91}
\begin{aligned}
\int_Q\Big( s \Theta   (w_x)^2 + s^3 \Theta^3
   \left(\frac{x-x_0}{k}\right)^2\ w^2 \Big)
    e^{2s \Gamma} \, dx da dt
       \le C\int_{Q}\frac{e^{2s \Gamma}}{k} F^2  dxdadt,
\end{aligned}
\end{equation}
for all $s \ge s_0$. Then, using the definition of $\xi$  and in
particular the fact that  $\xi_x$ and  $\xi_{xx}$ are supported in
$\hat \omega$, where  $\hat \omega:= (\bar \lambda_1, \tilde \lambda_1) \cup( \tilde \lambda_2, \bar \beta_2) $, we can write
\begin{equation}\label{fa}
\frac{F^2}{k}\le 2\frac{f^2}{k} + 2 k(\xi _{xx} v +2 \xi _x v_x)^2 \le  2\frac{f^2}{k}+ C( v^2+
(v_x)^2)\chi_{\hat \omega}.
\end{equation}
Hence, we find
\begin{equation}\label{stimacar1}
\begin{aligned}
&\int_0^T\int_0^A\int_{\tilde\lambda_1}^{\tilde \lambda_2}\left( s \Theta (v_x)^{2} + s^3
\Theta ^3 \left(\frac{x-x_0}{k}\right)^2 v^{2}\right)
e^{{2s\Gamma}} dxdadt\\
&=\int_0^T\int_0^A\int_{\tilde\lambda_1}^{\tilde \lambda_2}\Big( s \Theta   (w_x)^2 + s^3
\Theta^3 \left(\frac{x-x_0}{k}\right)^2 w^2 \Big)
e^{2s \Gamma} \, dxda dt\\
&\le C  \left( \int_0^T\int_0^A \int_{\hat \omega}e^{2s \Gamma}( v^2+
(v_x)^2)dxdadt+ \int_Q f^2 \frac{e^{2s\Gamma }}{k} dxdadt\right).\\
& \mbox{ (by Proposition \ref{caccio1}, since
$\hat\omega\subset \subset \omega$)}\\
&\le C \left( \int_{0}^T\int_0^A \int _{\omega}  \frac{ v^2 }{k} dxdadt + \int_Q\frac{f^2}{k} dxdadt\right).
\end{aligned}
\end{equation}
Now, consider a smooth function $\eta: [0,1] \to [0,1]$ such that
\[
\eta(x) =\begin{cases} 0& x\in [0,\bar\lambda_2],\\
1& x\in [\tilde\lambda_2, 1],
\end{cases}
\]
and define $z:= \eta v$. Then $z$ satisfies
\begin{equation}\label{eq-z*}
\begin{cases}
z_t + z_a+k  z_{xx}-\mu z= h,  &(t,x) \in Q_{T,A}\times (\lambda_2,1),\\
z(t,a, \lambda_2)= z(t,a, 1)=0, & t \in Q_{T,A},
\end{cases}
\end{equation}
with $h:=\eta f +k (\eta _{xx} v + 2\eta _x  v_x)\in L^2\big((0,T)\times
(\lambda_2, 1)\big)$. 
\\
Since the problem is {\em non degenerate} (observe that $x\in (\lambda_2,1)$) ,
we can apply Theorem \ref{nondegenereDebole}, with $(0,1)$ replaced by $(\lambda_2, 1)$ and Proposition
\ref{caccio1}, obtaining that there exist two  strictly positive constants $C$ and $s_0$ such that, for all $s \ge s_0$,
\begin{equation}\label{standa}
\begin{aligned}
&\int_0^T\int_0^A\int_{\lambda_2}^1 \left(s\Theta (z_x)^2 + s^3 \Theta^3
 z^2\right)e^{2s\Phi}dxdadt
\leq  C \!\int_0^T\!\int_0^A\!\int_{\lambda_2}^1 h^{2}e^{2s\Phi}dxdadt\\
&\le C \left(\!\int_0^T\!\int_0^A\!\int_{\tilde \omega}e^{2s\Phi}( v^2+ (v_x)^2)dxdadt+\!\! \int_Q f^2 \frac{e^{2s\Phi }}{k} dxdadt\right)\\
&\le C \left( \int_{0}^T\int_0^A \int _{\lambda_2}^{\beta_2}  \frac{ v^2 }{k} dxdadt + \int_Q f^2 \frac{e^{2s\Phi }}{k} dxdadt\right)\\
&\le C \left( \int_{0}^T\int_0^A \int _{\omega}  \frac{ v^2 }{k} dxdadt + \int_Q f^2 \frac{e^{2s\Phi }}{k} dxdadt\right),
\end{aligned}
\end{equation}
where $\tilde \omega =(\bar\lambda_2, \tilde\lambda_2)$ and  $\Phi$ is related to $  (\lambda_2,1)$. Observe that the boundary term which appears in the original estimate is nonpositive and thus is neglected.

Now, for a suitable choice of $d_1$ (see, for example, \cite{fm1}), there exists a  strictly positive constant $C$, 
such that
\begin{equation}\label{prima}
e^{2s\Gamma(t,x)} \le C e^{2s\Phi(t,x)}
\end{equation}
and
\begin{equation}\label{seconda}
\left(\frac{x-x_0}{k(x)}\right)^2e^{2s\Gamma(t,x)} \le C e^{2s
\Phi(t,x)}
\end{equation}
for all $(t,x) \in Q_{T,A}\times \left[\lambda_2, 1\right]$.
Thus, by \eqref{prima} and \eqref{seconda}, via \eqref{standa}, we
find
\[
\begin{aligned}
&\int_0^T\int_0^A\int_{\lambda_2}^1 \Big(s \Theta (z_x)^2
 + s^3 \Theta^3
     \left( \frac{x-x_0}{k}\right)^2z^2\Big) e^{2s\Gamma}dxdadt \\
& \le C \left( \int_{0}^T\int_0^A \int _{\omega}  \frac{ v^2 }{k} dxdadt + \int_Q f^2 \frac{e^{2s\Phi }}{k} dxdadt\right),
\end{aligned}
\]
for a  strictly positive constant $C$ and $s$ large enough. Hence, by
definition of $z$ and by the inequality above, we get
\begin{equation}\label{stimacar21}
\begin{aligned}
&\int_0^T\int_0^A\int_{\tilde\lambda_2}^1 \Big(s \Theta  (v_x)^2 + s^3 \Theta^3
\left(\frac{x-x_0}{k}\right)^2v^2\Big) e^{2s\Gamma} dxdadt\\
& \le \int_0^T\int_0^A\int_{\lambda_2}^1 \Big(s \Theta  (z_x)^2 + s^3 \Theta^3
\left( \frac{x-x_0}{k}\right)^2z^2\Big) e^{2s\Gamma}dxdadt
\\
&\le C \left( \int_{0}^T\int_0^A \int _{\omega}  \frac{ v^2 }{k} dxdadt + \int_Q f^2 \frac{e^{2s\Phi }}{k} dxdadt\right).
\end{aligned}
\end{equation}
Thus \eqref{stimacar1} and \eqref{stimacar21}  imply
\begin{equation}\label{carin01}
\begin{aligned}
&\int_0^T \int_0^A\int _{\tilde\lambda_1}^{1}  \Big( s \Theta  (v_x)^2 + s^3 \Theta^3
\left(\frac{x-x_0}{k}\right)^2  v^2 \Big) e^{2s \Gamma } \, dxda dt
\\
&\le C \left( \int_{0}^T\int_0^A \int _{\omega}  \frac{ v^2 }{k} dxdadt + \int_Q \frac{f^2}{k}dxdadt\right).
\end{aligned}
\end{equation}
To complete the
proof it is sufficient to prove a similar inequality for
$x\in[0,\tilde\lambda_1]$. To this aim, we use the reflection procedure
 of \cite{fm} or \cite{fm1}, considering the functions
\[
W(t,a,x):= \begin{cases} v(t,a,x), & x \in [0,1],\\
-v(t,a,-x), & x \in [-1,0],
\end{cases}
\]
\[
\tilde f(t,a,x):= \begin{cases} f(t,a,x), & x \in [0,1],\\
-f(t,a,-x), & x \in [-1,0],
\end{cases}
\]
and
\[
\tilde \mu(t,a,x):= \begin{cases} \mu(t,a,x), & x \in [0,1],\\
-\mu(t,a,-x), & x \in [-1,0],
\end{cases}
\]
so that $W$ satisfies the problem
\[
\begin{cases}
W_t +W_a +\tilde k W_{xx} - \tilde \mu W= \tilde f, &(t,x) \in  Q_{T,A}\times (-1,1),\\
W(t,a,-1)=W(t,a,1) =0, & t \in Q_{T,A}.
\end{cases}
\]
Here $\tilde k$ is as in \eqref{tildek}.

Now, consider a cut off function $\zeta: [-1,1] \to  [0,1]$ such that
\[\zeta (x) =\begin{cases} 0 & x\in[-1,-\bar \rho_1],\\
1& x\in [-\tilde\lambda_1, \tilde\lambda_1],\\
0&x\in [\bar \rho_1,1],
\end{cases}
\]
and define $Z:=\zeta W$. Then $Z$ satisfies
\begin{equation}\label{eq-Z*}
\begin{cases}
Z_t + Z_a+ \tilde kZ_{xx} -\tilde \mu Z=\tilde h,  &(t,x) \in Q_{T,A}\times (-\rho_1,\rho_1),\\
Z(t,a,-\rho_1)= Z(t,a,\rho_1)=0, & t \in Q_{T,A},
\end{cases}
\end{equation}
where $\tilde h=\zeta \tilde f+ \tilde k(\rho_{xx}W+2\tilde \rho_xW_x)$.
Now, applying 
 the analogue of Theorem \ref{nondegenereDebole} on $(- \rho_1,
\rho_1)$ in place of $(0,1)$, using the definition of $W$, the fact that $Z_x(t, a,-\rho_1)=Z_x(t,a,
\rho_1)=0$, analogous estimates of \eqref{prima} and \eqref{seconda} and since $\zeta$ is
supported in $\left[-\bar \rho_1, -\tilde \lambda_1\right] \cup\left[\tilde\lambda_1, \bar \rho_1\right]$, we get
\[
\begin{aligned}
& \int_0^T\int_0^A\int_{0}^{\tilde \lambda_1}  \left(s\Theta (W_x)^2 + s^3
\Theta^3
\left(\frac{x-x_0}{k} \right)^2W^2\right)e^{2s\Gamma}dxdadt\\
&= \int_0^T\int_0^A\int_{0}^{\tilde \lambda_1}  \left(s\Theta (Z_x)^2 + s^3
\Theta^3
\left(\frac{x-x_0}{k} \right)^2Z^2\right)e^{2s\Gamma}dxdadt\\
&\le C \int_0^T\int_0^A\int_{0}^{\rho_1}\left(s\Theta (Z_x)^2 + s^3 \Theta^3
Z^2\right)e^{2s\Phi}dxdadt\\
&\le C \int_0^T\int_0^A\int_{-\rho_1}^{\rho_1}\left(s\Theta (Z_x)^2 + s^3 \Theta^3
Z^2\right)e^{2s\Phi}dxdadt\\
& \le C
\int_0^T\int_0^A\int_{-\rho_1}^{\rho_1} \tilde h^{2}\frac{e^{2s\Phi}}{\tilde k}dxdadt\le C \int_0^T\int_0^A\int_{-\rho_1}^{\rho_1} \tilde f^{2}\frac{e^{2s\Phi}}{\tilde k}dxdadt \\
&+ C \int_0^T\int_0^A \int_{-\bar \rho_1}^{-\tilde\lambda_1}(
W^2+ (W_x)^2)e^{2s\Phi}dxdadt + C\int_0^T\int_0^A \int_{\tilde \lambda_1}^{ \bar \rho_1}(W^2+ (W_x)^2)e^{2s\Phi}dxdadt\\
&\le C \int_0^T\int_0^A\int_{-\rho_1}^{\rho_1} \frac{\tilde f^{2}}{\tilde k}dxdadt + C\int_0^T\int_0^A \int_{-\rho_1}^{- \lambda_1}W^2dxdadt +C\int_0^T\int_0^A\int_{\lambda_1}^{ \rho_1} W^2dxdadt \\
& \mbox{ (by Propositions \ref{caccio1} and since $\tilde f(t,a,x)= -f(t,a,-x)$, for $x <0$) }\\
& \le C \int_0^T\int_0^A\int_0^1 \frac{ f^2}{k}dxdadt +C\int_0^T\int_0^A\int_\omega v^2 dxdadt, 
\end{aligned}
\]
for some  strictly positive constants $C$ and $s$ large enough. Here  $\Phi$ is related to $  (-\rho_1,\rho_1)$.

 Hence, by definitions of $Z$, $W$ and $\zeta$, and using the previous inequality one has
\begin{equation}\label{car101}
\begin{aligned}
&\int_0^T\int_0^A\int_{0}^{\tilde\lambda_1} \left(s\Theta (v_x)^2 + s^3 \Theta^3
\left(\frac{x-x_0}{k} \right)^2v^2\right)e^{2s\Gamma}dxdadt\\
&= \int_0^T\int_0^A\int_{0}^{\tilde\lambda_1} \left(s\Theta (W_x)^2 + s^3 \Theta^3
\left(\frac{x-x_0}{k} \right)^2W^2\right)e^{2s\Gamma}dxdadt\\
&\le C\left( \int_Q \frac{f^{2}}{ k}dxdadt + \int_0^T\int_0^A \int_{\omega} \frac{v^2}{k}dxdadt\right).
\end{aligned}
\end{equation}
Moreover, by
\eqref{carin01} and \eqref{car101}, the conclusion follows.

Nothing changes in the proof if $\omega= \omega_1\cup \omega_2$ and each of these intervals lye on different sides of $x_0$, as the assumption implies.
\end{proof}
\begin{Remark}\label{remarkultimo}
Observe that the results of Theorems \ref{Cor2}, \ref{Cor12}, \ref{Cor13} still hold true if we substitute the interval $(0,T)$ with a general interval $(T_1,T_2)$, provided that $\mu$ and $\beta$ satisfy the required assumptions. In this case, in place of the function $\Theta$ defined in \eqref{571}, we have to consider the weight function
\[
\tilde \Theta(t,a):= \frac{1}{(t-T_1)^4 (T_2-t)^4a^4}.
\]
\end{Remark}
Using the previous local Carleman estimates one can prove the next observability inequalities.
\begin{Theorem}\label{Theorem4.4} Assume Hypothesis $\ref{BAss01}$ or $\ref{BAss02}$ and Hypothesis $\ref{conditionbeta}$ with $T > \bar a$. Suppose that $\omega$ satisfies \eqref{omega}. Then, for every $\delta \in (0,A)$, 
there exists a  strictly positive constant $C =C(\delta)$  such that every
solution $v$ of \eqref{h=0} in
$\mathcal V_1$
satisfies
\begin{equation}\nonumber
\begin{aligned}
 \int_0^A\int_0^1  \frac{1}{k}v^2( T-\bar a,a,x) dxda  &\le C\int_0^{T} \int_0^\delta \int_0^1 \frac{1}{k}v^2(t,a,x) dxdadt\\
 &+ C\left( \int_0^{T}\int_0^1 \frac{v_T^2(a,x)}{k}dxda+ \int_0^T \int_0^A\int_ \omega \frac{v^2}{k} dx dadt\right).
\end{aligned}\end{equation}
Moreover, if $v_T(a,x)=0$ for all $(a,x) \in (0,  T) \times (0,1)$, one has
\[
\begin{aligned}
 \int_0^A\!\!\int_0^1  \frac{1}{k}v^2(T -\bar a,a,x) dxda  &\le C\left(\int_0^{T} \int_0^\delta\!\! \int_0^1 \frac{1}{k}v^2(t,a,x) dxdadt +\!\! \int_0^T\!\! \int_0^A\!\!\int_ \omega \frac{v^2}{k} dx dadt\right).
\end{aligned}
\]
\end{Theorem}
\begin{proof}
Set
\begin{equation}\label{tildeT}
\tilde T:= T-\bar a.
\end{equation}
Using the method of characteristic lines, the assumption on $\beta$ and the fact that $v(t,A,x)=0$ for all $(t,x) \in Q_{T,1}$, one can compute
the following implicit formula for $v$ solution of \eqref{h=0}:
\begin{equation}\label{implicitformula}
S(T-t) v_T(T+a-t, \cdot),
\end{equation}
if $t \ge  \tilde T + a$ (observe that in this case $T+a-t \le \bar a$) and
\begin{equation}\label{implicitformula1}
v(t,a, \cdot)=\begin{cases}
S(T-t) v_T(T+a-t, \cdot)\!+\int_a^{T+a-t}S(s-a)\beta(s, \cdot)v(s+t-a, 0, \cdot) ds, &\Gamma\!= \!\bar a \\
\int_a^AS(s-a)\beta(s, \cdot)v(s+t-a, 0, \cdot) ds, & \Gamma\!= \!\Gamma_{A,T},
\end{cases}
\end{equation}
otherwise. Here  $(S(t))_{t \ge0}$ is the semigroup generated by the operator $\mathcal A_0 -\mu Id$ for all $u \in D(\mathcal A_0)$  ($Id$ is the identity operator), $\Gamma_{A,T}:= A -a +t-\tilde T$ and
\begin{equation}\label{Gamma}
\Gamma:= \min \{\bar a, \Gamma_{A,T}\}.
\end{equation}
In particular, it results
\begin{equation}\label{v(0)}
v(t,0, \cdot):= S(T-t) v_T(T-t, \cdot).
\end{equation}

Now, define, for $\varsigma >0$, the function $w= e^{\varsigma t }v$, where $v$ solves \eqref{h=0}.
Then $w$ satisfies
\begin{equation}\label{h=0'}
\begin{cases}
\ds \frac{\partial w}{\partial t} + \frac{\partial w}{\partial a}
+k(x)w_{xx}-(\mu(t, a, x)+ \varsigma) w =-\beta(a,x)w(t,0,x),& (t,x,a) \in  \tilde Q,
\\[5pt]
w(t,a,0)=w(t,a,1) =0, &(t,a) \in \tilde Q_{T,A},\\
  w(T,a,x) = e^{\varsigma T}v_T(a,x), &(a,x) \in Q_{A,1}, \\
  w(t,A,x)=0, & (t,x) \in \tilde Q_{T,1},
\end{cases}
\end{equation}
where $\tilde Q:= (\tilde T, T) \times Q_{A,1}$, $\tilde Q_{T,A}:= (\tilde T, T) \times (0,A)$ and $\tilde Q_{T,1}:= (\tilde T,T)\times (0,1)$.
Multiplying the equation of \eqref{h=0'} by $-\ds\frac{w}{k}$ and integrating by parts on $Q_t:= (\tilde T,t) \times(0,A) \times(0,1)$, it results
\begin{equation}\label{3.56idriss}
\begin{aligned}
&-\frac{1}{2}\int_{Q_{A,1}} \frac{1}{k}w^2(t,a,x) dxda + \frac{e^{\varsigma  \tilde T}}{2} \int_{Q_{A,1}}  \frac{1}{k}v^2(\tilde T,a,x) dxda + \frac{1}{2} \int_{\tilde T}^t\int_0^1  \frac{1}{k}w^2(\tau,0,x) dx d\tau\\
&+ \varsigma \int_{Q_t}  \frac{1}{k} w^2(\tau,a,x) dxdad\tau \le \int_{Q_t} \frac{1}{k}\beta w(\tau,0,x)wdxdad\tau
\\
& \le \|\beta\|_{L^\infty(Q)}\frac{1}{\epsilon}\int_{Q_t}\frac{1}{k}w^2dxdad\tau+ \epsilon A\|\beta\|_{L^\infty(Q)} \int_{\tilde T}^t\int_0^1 \frac{1}{k}w^2(\tau,0,x)dxd\tau,
\end{aligned}
\end{equation}
for $\epsilon >0$. Choosing $\ds\epsilon= \frac{1}{2\|\beta\|_{L^\infty(Q)}A}$ and $\ds \varsigma =\frac{ \|\beta\|_{L^\infty(Q)}}{\epsilon}$, we have
\[
\begin{aligned}
 \int_{Q_{A,1}}  \frac{1}{k}v^2(\tilde T,a,x) dxda  \le C\int_{Q_{A,1}} \frac{1}{k}w^2(t,a,x) dxda  \le C \int_{Q_{A,1}} \frac{1}{k}v^2(t,a,x) dxda.
\end{aligned}
\]
Now, take $\delta \in (0, A)$. Then, integrating over $\ds \left[\frac{T}{4}, \frac{3T}{4} \right]$,
\begin{equation}\label{t=0}
\begin{aligned}
 \int_{Q_{A,1}}  \frac{1}{k}v^2(\tilde T,a,x) dxda  &\le C\int_{\frac{T}{4}}^{\frac{3T}{4}} \int_{Q_{A,1}} \frac{1}{k}v^2(t,a,x) dxdadt\\
 &=  C \int_{\frac{T}{4}}^{\frac{3T}{4}} \left(\int_0^\delta + \int_\delta^A \right)\int_0^1 \frac{1}{k}v^2(t,a,x) dxdadt.
\end{aligned}
\end{equation}
Consider the term $\ds\int_{\frac{T}{4}}^{\frac{3T}{4}}  \int_\delta^A\int_0^1 \frac{1}{k}v^2(t,a,x) dxdadt$. By the Hardy - Poincar\'e inequality one has 
\begin{equation}\label{terminenuovo11}
\begin{aligned}
 \int_0^1\frac{v^2}{k}dx
& \le
C \int_0^1\frac{v^2}{x^2} dx
\le C \int_0^1 v_x^2 dx,
\end{aligned}\end{equation}
for a  strictly positive constant $C.$  Hence,
\[
\begin{aligned}
&\int_{\frac{T}{4}}^{\frac{3T}{4}}  \int_\delta^A\int_0^1 \frac{1}{k}v^2(t,a,x) dxdadt \le C \int_{\frac{T}{4}}^{\frac{3T}{4}}  \int_\delta^A\int_0^1 \Theta v_x^2e^{2s\varphi} dxdadt
\end{aligned}
\]
and, by Theorem \ref{Cor2} or \ref{Cor12},
\[
\int_{\frac{T}{4}}^{\frac{3T}{4}}  \int_\delta^A\int_0^1 \Theta v_x^2e^{2s\varphi} dxdadt\le
C\left(\int_{Q}\text{$\frac{f^2}{k}$\normalsize}~dxdadt+ \int_0^T \int_0^A\int_ \omega \frac{v^2}{k} dx dadt\right),
\]
where, in this case, $f(t,a,x):=-\beta(a,x)v(t,0,x)$.
Thus
\begin{equation}\label{t=01}
\int_{\frac{T}{4}}^{\frac{3T}{4}} \!\! \int_\delta^A\!\!\int_0^1 \frac{1}{k}v^2(t,a,x) dxdadt \le C\|\beta\|^2_{L^\infty(Q)}\!\left(\int_{Q}\frac{v^2(t,0,x)}{k}dxdadt+ \int_0^T \int_0^A\!\!\int_ \omega \frac{v^2}{k} dx dadt\right),
\end{equation}
for a  strictly positive constant $C$.
Now, using the fact that the semigroup generated by $\mathcal A_0$ is a contraction semigroup and the hypothesis on $\mu$, we have that also the semigroup generated by $\mathcal A_0 -\mu Id$ is bounded. Hence, by \eqref{v(0)},
\begin{equation}\label{t=02}
\int_{Q}\frac{v^2(t,0,x)}{k}dxdadt\le C\int_{Q_{T,1}} \frac{v_T^2(T-t,x)}{k}dxdt \le C \int_{Q_{T,1}} \frac{v_T^2(a,x)}{k}dxda.
\end{equation}
Hence, by \eqref{t=01} and \eqref{t=02}, one has
\begin{equation}\label{t=01'}
\int_{\frac{T}{4}}^{\frac{3T}{4}}\!\!  \int_\delta^A\!\!\int_0^1 \frac{1}{k}v^2(t,a,x) dxdadt \le C\|\beta\|^2_{L^\infty(Q)}\left(\int_{Q_{T,1}}\!\! \frac{v_T^2(a,x)}{k}dxda+ \int_0^T \!\!\int_0^A\!\!\int_ \omega \frac{v^2}{k} dx dadt\right),
\end{equation}
for a  strictly positive constant $C$.
From \eqref{t=0} and \eqref{t=01'}, it results
\begin{equation}\label{t=03}
\begin{aligned}
 \int_{Q_{A,1}}  \frac{1}{k}v^2(\tilde T,a,x) dxda  &\le C\int_0^T \int_0^\delta \int_0^1 \frac{1}{k}v^2(t,a,x) dxdadt\\
 &+ C\left( \int_{Q_{T,1}} \frac{v_T^2(a,x)}{k}dxda+ \int_0^T \int_0^A\int_ \omega \frac{v^2}{k} dx dadt\right).
\end{aligned}
\end{equation}

\end{proof} 
The observability inequality proved in the previous theorem still holds when $k$ degenerate at $x_0$:

\begin{Theorem}\label{CorOb1} Assume  Hypothesis $\ref{Ass021}$, Hypothesis $\ref{conditionbeta}$ and suppose that $\ds \frac{k'}{\sqrt k} \in L^\infty_{\text{loc}} ([0,1] \setminus\{x_0\})$ if Hypothesis $\ref{Ass0}$ holds with $T > \bar a$. Suppose that $\omega$ satisfies \eqref{omega0} or \eqref{omega_new}. Then, for every $\delta \in (0,A)$, 
there exists a  strictly positive constant $C =C(\delta)$ such that every
solution $v$ of \eqref{h=0} in
$\mathcal V_2$
satisfies
\begin{equation}\nonumber
\begin{aligned}
 \int_0^A\int_0^1  \frac{1}{k}v^2( T -\bar a,a,x) dxda  &\le C\int_0^{T} \int_0^\delta \int_0^1 \frac{1}{k}v^2(t,a,x) dxdadt\\
 &+ C\left( \int_0^{T}\int_0^1\frac{ v_T^2(a,x)}{k}dxda+ \int_0^T \int_0^A\int_ \omega \frac{v^2}{k} dx dadt\right).
\end{aligned}\end{equation}
Moreover, if $v_T(a,x)=0$ for all $(a,x) \in (0, T) \times (0,1)$, one has
\[
\begin{aligned}
 \int_0^A\!\!\int_0^1\!\!\  \frac{1}{k}v^2( T-\bar a,a,x) dxda  &\le C\left(\int_0^{T} \!\int_0^\delta\! \int_0^1 \!\frac{1}{k}v^2(t,a,x) dxdadt +\! \int_0^T\! \int_0^A\!\!\!\int_ \omega \!\frac{v^2}{k} dx dadt\right).
\end{aligned}
\]
\end{Theorem}
The proof of the previous inequalities follows the one of Theorem \ref{CorOb} so we omit it. But we underline the fact that, in order to obtain \eqref{terminenuovo11} in this situation, we distinguish the cases $M<1$ and $M\geq 1$. In the former case,
define
\[
p(x) =\frac{|x-x_0|^2}{k} \quad \text {and} \quad q=2-M.
\]
Clearly, by \cite[Lemma 2.1]{fm}, $p(x) \to 0$ as $x\to x_0$ and 
$
\ds x\mapsto \frac{p(x)}{|x-x_0|^q}=\frac{|x-x_0|^{M}}{k}
$ is nonincreasing on the left of $x=x_0$ and nondecreasing on the right of $x=x_0$. Moreover,
$q>1$ since $M<1$. Hence, by the Hardy-Poincar\'{e} inequality given in \cite[Proposition 2.3]{fm}, one has
\[
\begin{aligned}
\int_0^1 \frac{v^2}{k}dx &= \int_0^1
\frac{p(x)}{(x-x_0)^2}v^2dx
\le C_{HP}\int_0^1 p (v_x)^2 dx \\
&\le  C \int_0^1 (v_x)^2 dx,
\end{aligned}
\]
for a  strictly positive constant $C$. 

If $M \ge 1$, we can apply \cite[Lemma 3.7]{fggr} obtaining again
\[
\int_0^1
\frac{v^2}{k} dx \leq C\int_0^1 (v_x)^2dx,
\]
for a  strictly positive constant $C$.

Hence, in both cases \eqref{terminenuovo11} holds also if the degeneracy is in the interior of the domain. So, proceeding as before, we obtain the  thesis.

\begin{Corollary}\label{CorOb} Assume $\bar a = T$, Hypotheses $\ref{iposezione4}$ and $\ref{conditionbeta}$. Suppose that $\omega$ satisfies \eqref{omega}, \eqref{omega0} or \eqref{omega_new}. Then, , for every $\delta \in (0,A)$, 
there exists a  strictly positive constant $C =C(\delta)$
 such that every
solution $v$ of \eqref{h=0} in
$\mathcal V_i$, $i=1,2$,
satisfies
\begin{equation}\nonumber
\begin{aligned}
 \int_0^A\int_0^1  \frac{1}{k}v^2( 0,a,x) dxda  &\le C\int_0^T \int_0^\delta \int_0^1 \frac{1}{k}v^2(t,a,x) dxdadt\\
 &+ C\left( \int_0^T\int_0^1 \frac{v_T^2(a,x)}{k}dxda+ \int_0^T \int_0^A\int_ \omega \frac{v^2}{k} dx dadt\right).
\end{aligned}\end{equation}
Moreover, if $v_T(a,x)=0$ for all $(a,x) \in (0, T) \times (0,1)$, one has
\[
\begin{aligned}
 \int_0^A\!\!\int_0^1  \frac{1}{k}v^2(0,a,x) dxda  &\le C\left(\int_0^T \int_0^\delta\!\! \int_0^1 \frac{1}{k}v^2(t,a,x) dxdadt +\!\! \int_0^T\!\! \int_0^A\!\!\int_ \omega \frac{v^2}{k} dx dadt\right).
\end{aligned}
\]
\end{Corollary}
Actually, we can improve the previous results in the following way:
\begin{Theorem}\label{CorOb1'}Assume Hypotheses $\ref{iposezione4}$ and $\ref{conditionbeta}$. Suppose that $\omega$ satisfies \eqref{omega}, \eqref{omega0} or \eqref{omega_new}.  Then, for every $\delta \in (T,A)$, 
there exists a  strictly positive constant $C= C(\delta)$  such that every
solution $v$ of \eqref{h=0} in
$\mathcal V_i$, $i=1,2$,
satisfies
\begin{equation}\nonumber
 \int_0^A\int_0^1  \frac{1}{k}v^2(T-\bar a,a,x) dxda \le 
 C\left( \int_0^\delta \int_0^1  \frac{v_T^2(a,x)}{k}dxda+ \int_0^T \int_0^A\int_ \omega \frac{v^2}{k} dx dadt\right).
\end{equation}
\end{Theorem}
\begin{proof}
We distinguish between the two cases
$T=\bar a$ and $T>\bar a$.

\underline{\it If $T=\bar a$:} Taking $\delta \in (T,A)$, one has, as in \eqref{t=0},
\begin{equation}\label{t=04}
 \int_{Q_{A,1}}  \frac{1}{k}v^2(0,a,x) dxda  \le C\int_{\frac{T}{4}}^{\frac{3T}{4}} \left(\int_0^{\delta-\frac{3T}{4}} + \int_{\delta-\frac{3T}{4}}^A \right)\int_0^1 \frac{1}{k}v^2(t,a,x) dxdadt.
\end{equation}
As for \eqref{t=01'},
\begin{equation}\label{t=05}
\begin{aligned}
\int_{\frac{T}{4}}^{\frac{3T}{4}}  \int_{\delta-\frac{3T}{4}}^A\int_0^1 \frac{1}{k}v^2(t,a,x) dxdadt \le C\left( \int_{Q_{T,1}} \frac{v_T^2(a,x)}{k}dxda+ \int_0^T \int_0^A\int_ \omega \frac{v^2}{k} dx dadt\right).
\end{aligned}
\end{equation}
Now, consider the term $\ds \int_{\frac{T}{4}}^{\frac{3T}{4}}  \int_0^{\delta-\frac{3T}{4}}\int_0^1 \frac{1}{k}v^2(t,a,x) dxdadt $ and let us prove that there exists $C>0$ such that
\begin{equation}\label{t=06}
\ds \int_{\frac{T}{4}}^{\frac{3T}{4}}  \int_0^{\delta-\frac{3T}{4}}\int_0^1 \frac{1}{k}v^2(t,a,x) dxdadt  \le C \int_0^\delta \int_0^1  \frac{v_T^2(a,x)}{k}dxda.
\end{equation}
In order to prove \eqref{t=06}, we use \eqref{implicitformula} or \eqref{implicitformula1}. 
Observe, first of all, that $\ds \delta -  \frac{3T}{4} >  \frac{T}{4}$, but we do not know if $\ds \delta -  \frac{3T}{4} \ge  \frac{3T}{4}$ or $\ds \delta -  \frac{3T}{4} <  \frac{3T}{4}$.
In the last case, if 
$t \in \left[\delta - \ds \frac{3T}{4},  \frac{3T}{4}\right)$ and $ a \in \left(0, \delta - \ds \frac{3T}{4}\right)$, we have easily that  $t \ge \tilde T+a = a$ (recall that we are in the case $\tilde T = 0$); hence \eqref{implicitformula} holds.
On the other hand, if $t <\ds \delta -  \frac{3T}{4}$ we do not know if $t \ge \tilde T+a = a$ or $t < \tilde T+a = a$. Hence, we have to consider  \eqref{implicitformula} or \eqref{implicitformula1}.
Taking into account these considerations, using the assumption on $\beta$ and the boundedness of $(S(t))_{t \ge0}$, one has:

{\it If $\delta - \ds \frac{3T}{4}<\frac{3T}{4}$:}
\begin{equation}\label{t=07}
\begin{aligned}
&\ds \int_{\frac{T}{4}}^{\frac{3T}{4}}  \int_0^{\delta-\frac{3T}{4}}\int_0^1 \frac{1}{k}v^2(t,a,x) dxdadt =
&\ds \int_{\frac{T}{4}}^{\delta - \frac{3T}{4}}  \int_0^{\delta-\frac{3T}{4}}\int_0^1 \frac{1}{k}v^2(t,a,x) dxdadt \\
&+ \ds \int_{\delta - \frac{3T}{4}}^{\frac{3T}{4}}  \int_0^{\delta-\frac{3T}{4}}\int_0^1 \frac{1}{k}v^2(t,a,x) dxdadt.
\end{aligned}
\end{equation}
By \eqref{implicitformula},
\begin{equation}\label{t=071}
\begin{aligned}
 &\ds \int_{\delta - \frac{3T}{4}}^{\frac{3T}{4}}  \int_0^{\delta-\frac{3T}{4}}\int_0^1 \frac{1}{k}v^2(t,a,x) dxdadt\le C  \int_{\delta - \frac{3T}{4}}^{\frac{3T}{4}}  \int_0^{\delta-\frac{3T}{4}}\int_0^1  \frac{v_T^2(T+a-t,x)}{k}dxdadt\\
 &\le C  \int_{\frac{T}{4}}^{\frac{3T}{4}}  \int_0^{\delta-\frac{3T}{4}}\int_0^1  \frac{v_T^2(T+a-t,x)}{k}dxdadt \\
&\le C  \int_{\frac{T}{4}}^{\frac{3T}{4}}  \int_0^{\delta-\frac{3T}{4}}\int_0^1   \frac{v_T^2(a+z,x)}{k}dxdadz  \le  C \int_0^{\delta-\frac{3T}{4}}\int_{a+\frac{T}{4}}^{a+\frac{3T}{4}} \int_0^1   \frac{v_T^2(\sigma,x)}{k}dxd\sigma da
\\
&\le C \int_0^\delta \int_0^1  \frac{v_T^2(\sigma,x)}{k}dxd\sigma.
\end{aligned}
\end{equation}
Consider now the integral
\[
\ds \int_{\frac{T}{4}}^{\delta - \frac{3T}{4}}  \int_0^{\delta-\frac{3T}{4}}\int_0^1 \frac{1}{k}v^2(t,a,x) dxdadt \]
and divide it in  the following way:
\begin{equation}\label{t=08}
\begin{aligned}
\ds \int_{\frac{T}{4}}^{\delta - \frac{3T}{4}}  \int_0^{\delta-\frac{3T}{4}}\int_0^1 \frac{1}{k}v^2(t,a,x) dxdadt
&=  \int_{\frac{T}{4}}^{\delta - \frac{3T}{4}} \int_0^{t}\int_0^1 \frac{1}{k}v^2(t,a,x) dxdadt\\
&+  \int_{\frac{T}{4}}^{\delta - \frac{3T}{4}} \int_t^{\delta-\frac{3T}{4}}\int_0^1 \frac{1}{k}v^2(t,a,x) dxdadt.
\end{aligned}
\end{equation}
Proceeding as before, one can prove 
\begin{equation}\label{t=081}
\int_{\frac{T}{4}}^{\delta - \frac{3T}{4}} \int_0^t\int_0^1 \frac{1}{k}v^2(t,a,x) dxdadt\le C  \int_0^\delta \int_0^1  \frac{v_T^2(\sigma,x)}{k}dxd\sigma.
\end{equation}
Indeed, since $\tilde T=0$ and $a\le t$, \eqref{implicitformula} holds, hence
\[\begin{aligned}
&\ds \int_{\frac{T}{4}}^{\delta - \frac{3T}{4}} \int_0^t\int_0^1 \frac{1}{k}v^2(t,a,x) dxdadt\le C  \int_{\frac{T}{4}}^{\frac{3T}{4}}  \int_0^t\int_0^1  \frac{v_T^2(T+a-t,x)}{k}dxdadt \\
&\le C  \int_{\frac{T}{4}}^{\frac{3T}{4}}  \int_0^{T-z}\int_0^1   \frac{v_T^2(a+z,x)}{k}dxdadz  \le  C \int_{\frac{T}{4}}^{\frac{3T}{4}} \int_0^T \int_0^1   \frac{v_T^2(\sigma,x)}{k}dxd\sigma dz
\\
&\le C \int_0^\delta \int_0^1  \frac{v_T^2(\sigma,x)}{k}dxd\sigma.
\end{aligned}\]
Now, we estimate the second term in the right hand side  of \eqref{t=08}. First of all, assume that $\Gamma=\bar a$ (we recall that $\Gamma$ is defined in \eqref{Gamma}).  By \eqref{implicitformula1} and \eqref{v(0)}, using the assumption on $\beta$ and the boundedness of $(S(t))_{t \ge0}$, one has:
\begin{equation}\label{t=082}
\begin{aligned}
&\ds \int_{\frac{T}{4}}^{\delta - \frac{3T}{4}}  \int_t^{\delta-\frac{3T}{4}}\int_0^1 \frac{1}{k}v^2(t,a,x) dxdadt\le C\int_{\frac{T}{4}}^{\delta - \frac{3T}{4}}  \int_t^{\delta-\frac{3T}{4}}\int_0^1   \frac{v_T^2(T+a-t,x)}{k}dxdadt \\
&+  C  \int_{\frac{T}{4}}^{\delta - \frac{3T}{4}}  \int_t^{\delta-\frac{3T}{4}}\int_0^1    \left(\int_a^{T+a-t} \frac{v_T^2(T-s-t+a,x)}{k}ds \right) dxdadt \\
& \le C\int_{\frac{T}{4}}^{\frac{3T}{4}}  \int_0^{\delta-\frac{3T}{4}}\int_0^1   \frac{v_T^2(T+a-t,x)}{k}dxdadt \\
& +  C  \int_{\frac{T}{4}}^{\frac{3T}{4}}  \int_t^{\delta-\frac{3T}{4}}\int_0^1    \left(\int_{-a}^{T-t-a} \frac{v_T^2(a+z,x)}{k}dz\right) dxdadt\\
& (\text{proceeding as in \eqref{t=07}  for the first integral})\\
& \le   C\int_{0}^{\delta} \int_0^1   \frac{v_T^2(\sigma,x)}{k}dxd\sigma +  C  \int_{\frac{T}{4}}^{\frac{3T}{4}}  \int_t^{\delta-\frac{3T}{4}}\int_0^1    \left(\int_0^{T-t}\frac{v_T^2(\sigma,x)}{k}d\sigma\right) dxdadt\\
& (\text{since $T < \delta$})\\
& \le  C\int_{0}^{\delta} \int_0^1   \frac{v_T^2(\sigma,x)}{k}dxd\sigma +  C  \int_{\frac{T}{4}}^{\frac{3T}{4}}  \int_0^{\delta-\frac{3T}{4}}\int_0^1    \left(\int_0^\delta \frac{v_T^2(\sigma,x)}{k}d\sigma\right) dxdadt\\
& \le  C\int_{0}^{\delta} \int_0^1   \frac{v_T^2(\sigma,x)}{k}dxd\sigma.
\end{aligned}
\end{equation}
Now, assume that $ \Gamma= A+\bar a -a +t-T$ (this implies that $A-a \le T-t$). By \eqref{implicitformula1} and \eqref{v(0)},  one has, as before:
\begin{equation}\label{t=09}
\begin{aligned}
&\ds \int_{\frac{T}{4}}^{\delta - \frac{3T}{4}}  \int_t^{\delta-\frac{3T}{4}}\!\!\int_0^1 \frac{1}{k}v^2(t,a,x) dxdadt\\
&\le  C  \!\!\int_{\frac{T}{4}}^{\delta - \frac{3T}{4}}  \int_t^{\delta-\frac{3T}{4}}\!\!\int_0^1    \left(\int_a^A \frac{v_T^2(T-s-t+a,x)}{k}ds \right) dxdadt \\
& \le  C  \int_{\frac{T}{4}}^{\frac{3T}{4}}  \int_t^{\delta-\frac{3T}{4}}\int_0^1    \left(\int_{T-A-t}^{T-a-t} \frac{v_T^2(a+z,x)}{k}dz\right) dxdadt\\
& \le   C  \int_{\frac{T}{4}}^{\frac{3T}{4}}  \int_t^{\delta-\frac{3T}{4}}\int_0^1    \left(\int_{T-t-(A-a)}^{T-t}\frac{v_T^2(\sigma,x)}{k}d\sigma\right) dxdadt\\
& (\text{since $T < \delta$})\\
& \le C  \int_{\frac{T}{4}}^{\frac{3T}{4}}  \int_0^{\delta-\frac{3T}{4}}\int_0^1    \left(\int_0^\delta \frac{v_T^2(\sigma,x)}{k}d\sigma\right) dxdadt\\
& \le  C\int_{0}^{\delta} \int_0^1   \frac{v_T^2(\sigma,x)}{k}dxd\sigma.
\end{aligned}
\end{equation}
Hence, in every case \eqref{t=06} holds.

By \eqref{t=04}, \eqref{t=05} and \eqref{t=06}, it follows that
\[
\begin{aligned}
 &\int_{Q_{A,1}}  \frac{1}{k}v^2(0,a,x) dxda  \le C\int_{\frac{T}{4}}^{\frac{3T}{4}} \left(\int_0^{\delta-\frac{3T}{4}} + \int_{\delta-\frac{3T}{4}}^A \right)\int_0^1 \frac{1}{k}v^2(t,a,x) dxdadt\\
 &\le C\left( \int_0^\delta \int_0^1  \frac{v_T^2(a,x)}{k}dxda+ \int_0^T \int_0^A\int_ \omega \frac{v^2}{k} dx dadt\right).
\end{aligned}
\]

{\it If $\delta - \ds \frac{3T}{4}\ge\frac{3T}{4}$:} In order to obtain \eqref{t=06}, we divide the integral
\[
\ds \int_{\frac{T}{4}}^{\frac{3T}{4}}  \int_0^{\delta-\frac{3T}{4}}\int_0^1 \frac{1}{k}v^2(t,a,x) dxdadt \]
in  the following way:
\begin{equation}\label{t=08new}
\begin{aligned}
\ds \int_{\frac{T}{4}}^{\frac{3T}{4}}  \int_0^{\delta-\frac{3T}{4}}\int_0^1 \frac{1}{k}v^2(t,a,x) dxdadt
&=  \int_{\frac{T}{4}}^{\frac{3T}{4}} \int_0^{t}\int_0^1 \frac{1}{k}v^2(t,a,x) dxdadt\\
&+  \int_{\frac{T}{4}}^{ \frac{3T}{4}} \int_t^{\delta-\frac{3T}{4}}\int_0^1 \frac{1}{k}v^2(t,a,x) dxdadt.
\end{aligned}
\end{equation}

Then, proceeding as before, the thesis follows.

\underline{\it If $T>\bar a$:} We proceed as before substituting, for simplicity, $\ds \frac{T}{4}$ and $\ds \frac{3T}{4}$ with $T-\bar a$ and $T- \ds \frac{\bar a}{4}$, respectively. 
In particular, taking $\delta \in (T,A)$,  we will consider, in place of \eqref{t=04}, the following inequality:
\begin{equation}\label{t=041}
 \int_{Q_{A,1}}  \frac{1}{k}v^2(T-\bar a,a,x) dxda  \le C\int_{T-\bar a}^{T- \frac{\bar a}{4}} \left(\int_0^{\delta - \bar a} + \int_{\delta - \bar a}^A \right)\int_0^1 \frac{1}{k}v^2(t,a,x) dxdadt.
\end{equation}
Also in this case, since $t \in \ds \left( T-\bar a, T- \frac{\bar a}{4}\right)$ and $a \in (0, \delta - \bar a)$, we do not know  if $t \ge \tilde T+ a$ or $t < \tilde T + a$.  Hence, 
to prove an estimate like \eqref{t=06}, 
we have to consider different cases as before.

\end{proof}

\vspace{0.4cm}
By Theorem \ref{CorOb1'} and using a density argument, one can prove
 Proposition \ref{obser.}. As a consequence one has the following null controllability results:

\begin{Theorem}\label{NCintermediate} Assume Hypotheses $\ref{iposezione4}$ and $\ref{conditionbeta}$. Then, given $T>0$ and $y_0 \in L^2_{\frac{1}{k}}(Q_{A,1})$,  for every $\delta \in (T,A)$  there exists a control $f\in L^2_{\frac{1}{k}}(\tilde Q)$ such that the solution $y\in \mathcal U$ of 
\begin{equation} \label{1new}
\begin{cases}
\ds \frac{\partial y}{\partial t}  +\frac{\partial y}{\partial a}
-k(x)y_{xx}+\mu(t, a, x)y =f(t,x,a)\chi_{\omega}  &\quad \text{in } \tilde Q,\\
  y(t, a, 1)=y(t, a, 0)=0  &\quad \text{on } \tilde Q_{T,A},\\
 y(\tilde T , a, x)=y_0(a, x) &\quad \text{in }Q_{A,1},\\
 y(t, 0, x)=\int_0^A \beta (a, x)y (t, a, x) da  &\quad  \text{in } \tilde Q_{T,1},
\end{cases}
\end{equation}
satisfies
\[
y(T,a,x) =0 \quad \text{a.e. } (a,x) \in (\delta, A) \times (0,1).
\]
Moreover, there exists $C= C(\delta)>0$ such that
\begin{equation}\label{stimaf1}
\|f\|_{L^2_{\frac{1}{k}}(\tilde Q)} \le C \|y_0\|_{L^2_{\frac{1}{k}}(Q_{A,1})}.
\end{equation}
Here, we recall, $\tilde Q=(\tilde T,T)\times(0,A)\times(0,1)$, $\tilde Q_{T,A} = (\tilde T,T)\times (0,A)$
and $\tilde Q_{T,1}=(\tilde T,T)\times(0,1)$.
\end{Theorem}
\begin{proof}
Take $g \in L^2_{\frac{1}{k}}(Q_{A,1})$ such that $g(A,x)=0$ in $(0,1)$ and fix $\delta \in (T,A)$. Let $v$ be the solution of 
\begin{equation}\label{ad}
\begin{cases}
\ds \frac{\partial v}{\partial t} + \frac{\partial v}{\partial a}
+k(x)v_{xx}-\mu(t, a, x)v =-\beta(t,x,a)v(t,0,x),& (t,x,a) \in \tilde  Q,
\\[5pt]
v(t,a,0)=v(t,a,1) =0, &(t,a) \in \tilde Q_{T,A},\\
  v(T,a,x) = v_T(a,x):= \begin{cases}g(a,x), & (a,x) \in (\delta, A) \times (0,1),\\
  0, & (a,x) \in (0, \delta) \times (0,1),
  \end{cases} \\
  v(t,A,x)=0, & (t,x) \in \tilde Q_{T,1}.
\end{cases}\end{equation}
Now, fixed  $y_0 \in L^2_{\frac{1}{k}}(Q_{A,1})$, define as in \cite{fcm},
\[
J(g) = \frac{1}{2}\int_{\tilde T}^T \int_0^A\int_\omega \frac{v^2}{k}  dxda dt+ \int_0^A\int_0^1 \frac{1}{k}v(\tilde T,a, x) y_0(a,x) dxda.
\]
The functional $J$ is strictly convex, continuous and coercive over the Hilbert space $\mathcal H$ defined by the completion of $L^2((\delta, A) \times (0,1))$ with respect to the norm $\|v\|_{L^2(\tilde Q_{T,A}\times \omega)}$. Thus, there exists a unique minimum, $\hat g$, of $J$  and $\hat g(A,x)=0$ in $(0,1)$. Let $\hat v$ be the solution of \eqref{ad} associated to $\hat g$. Define 
$
f:=\hat v\chi_\omega
$
and let $y$ be the solution of \eqref{1new} in $\tilde Q$ associated to $f$. Since $\hat g$ is the minimum of $J$, it results
\begin{equation}\label{*2}
\begin{aligned}
0=\left[ \frac{d}{dt} J(\hat g + t g)\right]_{t=0} = \int_{\tilde T}^T\int_0^A \int_\omega \frac{1}{k} v \hat v  dxdadt  + \int_0^A\int_0^1 \frac{1}{k}v({\tilde T},a, x) y_0(a, x)  dxda,
\end{aligned}
\end{equation}
for all $g \in L^2(Q_{A,1})$  such that $g(A,x)=0$ in $(0,1)$. In particular, for $g = \hat g$, one has
\[
0 = \int_{\tilde T}^T \int_0^A\int_\omega \frac{\hat v^2}{k} dxda dt +  \int_0^A\int_0^1 \frac{1}{k} \hat v({\tilde T},a, x) y_0(a, x) dxda.
\]
Hence
\begin{equation}\label{a2}
\int_{\tilde T}^T \int_0^A\int_\omega  \frac{\hat v^2}{k} dxda dt =- \int_0^A\int_0^1\frac{1}{k} \hat v({\tilde T},a,x) y_0(a,x)  dxda,
\end{equation}
and, by H\"older's inequality, by Proposition \ref{obser.} applied to $\hat v$ in $\tilde Q$ and using the fact that $v_T(a,x) =0$ for all $(a,x) \in (0, \delta)\times (0,1)$, one has
\begin{equation}\label{a3}
\begin{aligned}
&\left| \int_0^A\int_0^1\frac{1}{k} \hat v({\tilde T},a, x) y_0(a, x)  dxda \right| \le \left(\int_0^A\int_0^1\frac{1}{k} \hat v^2({\tilde T},a, x)  dxda\right)^{\frac{1}{2}} \left(\int_0^A\int_0^1 \frac{1}{k}y_0^2(a,x) dxda\right)^{\frac{1}{2}}\\
& \le C\left(\int_{\tilde T}^T \int_0^A\int_ \omega \frac{\hat v^2}{k} dx dadt\right)^{\frac{1}{2}} \|y_0\|_{L^2_{\frac{1}{k}}(Q_{A,1})}.
\end{aligned}
\end{equation}
Thus, by \eqref{a2} and \eqref{a3},
\begin{equation}\label{a4}
\begin{aligned}
 \int_{\tilde T}^T\int_0^A\io \frac{1}{k}\hat v^2(t,a, x) dxda dt\le  C \left(\int_{\tilde T}^T \int_0^A\int_ \omega \frac{\hat v^2}{k} dx dadt\right)^{\frac{1}{2}}\|y_0\|_{L^2_{\frac{1}{k}}(Q_{A,1})}.
\end{aligned}
\end{equation}
Hence
\[
\|f\|_{L^2_{\frac{1}{k}}(\tilde Q)} = \left( \int_{\tilde T}^T\int_0^A\io \frac{\hat v^2}{k}dxda dt\right)^{\frac{1}{2}} \le C\|y_0\|_{L^2_{\frac{1}{k}}(Q_{A,1})}.
\]
Now, let $y$ be the solution of \eqref{1new} asssociated to $f$ and $y_0$.
\\
Multiplying the equation of \eqref{ad} by $\ds\frac{y}{k}$ and integrating over $\tilde Q$, one has:
\[
\begin{aligned}
0& = \int_{\tilde Q} \left(\frac{\partial v}{\partial t} + \frac{\partial v}{\partial a}
+k(x)v_{xx}-\mu(t, a, x)v +\beta(x,a)v(t,0,x)\right)  \frac{y}{k}dxda dt \Longleftrightarrow \\
0& = \int_\delta^A \int_0^1\frac{1}{k}y(T,a,x)g(a,x)\  dxda -\int_{Q_{A,1}}\frac{1}{k} y_0 v({\tilde T},a,x)  dxda - \int_{\tilde Q_{T,1}}\frac{1}{k} y(t,0,x) v(t,0,x)  dxdt\\
& +\int_{\tilde Q} \frac{1}{k}  \beta(a,x) v(t,0,x) y(t,a,x) dxdadt- \int_{\tilde Q}\frac{v}{k}\left(\frac{\partial y}{\partial t} + \frac{\partial y}{\partial a}
-k(x)y_{xx}+\mu(t, a, x)y\right)   dxda dt
\end{aligned}
\]
(recall that $y(t, 0, x)=\int_0^A \beta (a, x)y (t, a, x) da$).
But $\ds\frac{\partial y}{\partial t} + \frac{\partial y}{\partial a}
-k(x)y_{xx}+\mu(t, a, x)y= f\chi_\omega$; hence
\[
\begin{aligned}
0&=\int_\delta^A \int_0^1\frac{1}{k}y(T,a,x)g(a,x)\  dxda -\int_{Q_{A,1}}\frac{1}{k} y_0 v({\tilde T},a,x)  dxda- \int_{\tilde T}^T\int_0^A \int_\omega\frac{v\hat v}{k}  dxda dt.
\end{aligned}
\]
Thus, being by \eqref{*2}
\[
\int_{\tilde T}^T\int_0^A \int_\omega \frac{1}{k} v \hat v  dxdadt  =- \int_0^A\int_0^1 \frac{1}{k}v({\tilde T},a, x) y_0(a, x)  dxda,
\]
it follows
\[
0=\int_\delta^A \int_0^1\frac{1}{k}y(T,a,x)g(a,x)\  dxda
\]
for all $g \in L^2_{\frac{1}{k}}(Q_{A,1})$  with $g(A,x)=0$ in $(0,1)$. Hence  $y(T,a,x)=0$ a.e. $(a,x) \in (\delta, A) \times(0,1)$.

\end{proof}

Observe that if $T= \bar a$, Theorem  \ref{NCintermediate} is exactly the null controllability result that we expect. Indeed, in this case  \eqref{1new} coincide with \eqref{1}. On the other hand, if $T>\bar a$, the null controllability for \eqref{1} is given in the next theorem and it is based on the previous result:

\begin{Theorem}\label{ultimo} Assume Hypotheses $\ref{iposezione4}$ and $\ref{conditionbeta}$. Suppose that $\omega$ satisfies \eqref{omega}, \eqref{omega0} or \eqref{omega_new}. Then, given $T>0$ and $y_0 \in L^2_{\frac{1}{k}}(Q_{A,1})$,  for every $\delta \in (T,A)$,   there exists a control $f\in L^2_{\frac{1}{k}}(Q)$ such that the solution $y$ of  \eqref{1}
satisfies
\[
y(T,a,x) =0 \quad \text{a.e. } (a,x) \in (\delta, A) \times (0,1).
\]
Moreover, there exists $C>0$ such that
\begin{equation}\label{stimaf}
\|f\|_{L^2_{\frac{1}{k}}( Q)} \le C \|y_0\|_{L^2_{\frac{1}{k}}(Q_{A,1})}.
\end{equation}
\end{Theorem}
\begin{proof} Set $\tilde T := T-\bar a \in (0,T)$.
By Theorem \ref{theorem_nondivergence}, there exists a unique solution $u$  of
\begin{equation} \label{1new2}
\begin{cases}
\ds \frac{\partial u}{\partial t}  +\frac{\partial u}{\partial a}
-k(x)u_{xx}+\mu(t, a, x)u =0  &\quad \text{in } (0, \tilde T) \times(0,A)\times (0,1),\\
  u(t, a, 1)=u(t, a, 0)=0  &\quad \text{on }  (0, \tilde T) \times(0,A),\\
 u(0 , a, x)=y_0(a, x) &\quad \text{in } (0,A)\times (0,1),\\
 u(t, 0, x)=\int_0^A \beta (a, x)u (t, a, x) da &\quad  \text{in }(0, \tilde T) \times(0,1).
\end{cases}
\end{equation}
Set $\tilde y_0(a,x) :=u(\tilde T, a, x)$; clearly $\tilde y_0 \in L^2_{\frac{1}{k}}(Q_{A,1})$.
Now, consider 
\begin{equation} \label{1new1}
\begin{cases}
\ds \frac{\partial w}{\partial t}  +\frac{\partial w}{\partial a}
-k(x)w_{xx}+\mu(t, a, x)w =h(t,x,a)\chi_{\omega} & \quad \text{in } \tilde Q,\\
  w(t, a, 1)=w(t, a, 0)=0&  \quad \text{on } \tilde Q_{T,A},\\
 w(\tilde T , a, x)=\tilde y_0(a, x) &\quad \text{in }Q_{A,1},\\
 w(t, 0, x)=\int_0^A \beta (a, x)w (t, a, x) da & \quad  \text{in } \tilde Q_{T,1}.
\end{cases}
\end{equation}
Again, by Theorem \ref{theorem_nondivergence}, there exists a unique solution $w $  of \eqref{1new1} and, by the previous Theorem, there exists a control $h \in L^2_{\frac{1}{k}}(\tilde Q)$ such that
\[
w(T,a,x)=0 \text{ a.e. } (a,x) \in (\delta, A) \times(0,1).
\]
Now, define $y$ and $f$ by
\[
y:= \begin{cases} u, & \text{in}\; [0, \tilde T],\\
w, & \text{in} \; [\tilde T, T] \end{cases} \quad \text{and}\quad f:=\begin{cases}0, & \text{in}\; [0, \tilde T], \\ h, & \text{in}\; [\tilde T, T]. \end{cases}
\]
Then $y$ satisfies \eqref{1} and $f \in L^2_{\frac{1}{k}}(Q)$ is such that
\[
y(T,a,x)=0 \text{ a.e. } (a,x) \in (\delta, A) \times(0,1).
\]
Indeed $y(T,a,x) = w(T,a,x) =0$ a.e. $(a,x) \in (\delta, A) \times (0,1)$.\\
It remains to prove \eqref{stimaf}. To this aim, observe that, by \eqref{stimaf1},
\begin{equation}\label{ultima1}
\|f\|_{L^2_{\frac{1}{k}}(Q)}^2 = \int_{\tilde T}^T\int_0^A \int_0^1 \frac{h^2}{k} dxdadt  \le C \|\tilde y_0\|_{L^2_{\frac{1}{k}}(Q_{A,1})}^2 =C \int_0^A \int_0^1 \frac{u^2}{k}(\tilde T, a,x)dxda
\end{equation}
for a  strictly positive constant $C$. Thus, it is sufficient to estimate the last integral. To do this,  we multiply  the equation of \eqref{1new2} by $\ds \frac{u}{k}$  and we integrate over $Q_{A,1}$, obtaining:
\[
\frac{1}{2}\frac{d}{dt} \int_0^A \int_0^1 \frac{u^2}{k} dxda + \frac{1}{2}\int_0^1 \frac{u^2(t,A,x)}{k}dx + \int_0^A\int_0^1 u_x^2 dxda + \int_0^A \int_0^1 \mu \frac{u^2}{k} dxda=  \frac{1}{2}\int_0^1 \frac{u^2(t,0,x)}{k}dx.
\]
Hence, using the fact that $u(t,0,x)= \int_0^A \beta(a,x) u(t,a,x)da$,
\[
\frac{1}{2}\frac{d}{dt} \int_0^A \int_0^1 \frac{u^2}{k} dxda \le \frac{1}{2}\int_0^1 \frac{1}{k}\left(\int_0^A \beta(a,x) u(t,a,x)da\right)^2dx\le  \frac{C}{2}\int_0^A\int_0^1 \frac{u^2}{k}dxda.
\]
Setting $F(t):= \|u(t)\|^2_{L^2_{\frac{1}{k}}(Q_{A,1})}$ and multiplying the previous inequality by $e^{-Ct}$, it results
\[
\frac{d}{dt} \left(e^{-Ct}F(t)\right) \le 0.
\]
For all $t \in (0,T),$ integrating over $(0,t)$, we have
\[
\int_0^A \int_0^1 \frac{u^2(t,a,x)}{k} dxda \le C \int_0^A \int_0^1 \frac{u^2(0,a,x)}{k} dxda=C \int_0^A \int_0^1 \frac{y_0^2(a,x)}{k} dxda.
\]
In particular,
\begin{equation}\label{ultima2}
\int_0^A \int_0^1 \frac{u^2(\tilde T,a,x)}{k} dxda\le C\int_0^A \int_0^1 \frac{y_0^2(a,x)}{k} dxda.
\end{equation}
By \eqref{ultima1} and \eqref{ultima2}, \eqref{stimaf} follows.
\end{proof}

Actually, in the (ISD) case, this result can be deduced directly by Theorem \ref{ultimo} in the (BD) case. Indeed, it holds also if we substitute the space interval $(0,1)$ with a general interval $(\mathcal A, \mathcal B)$ provided that $k$ satisfies the required assumptions in this interval.
Now,  if we are in the (ISD) case, by \cite[Proposition 3.6]{fggr}, $y(t,a,x_0)=0$ a.e. $(t,a) \in Q_{T,A}$; hence, we can divide  \eqref{1} into two problems stated in $Q_{T,A} \times (0,x_0)$ and in $Q_{T,A}\times(x_0,1)$, respectively, and we can apply Theorem \ref{ultimo} in the (BD) case, obtaining the thesis.
This technique does not work in the weakly degenerate case since we are not able to divide the problem into two disjoint systems due the lack of the
characterization of $H^1_{\frac{1}{k}}(0,1)$. However, using observability inequalities and Carleman estimates, we are able to prove a null controllability result also in this case.

\section{Appendix}
\paragraph{Proof of Theorem \ref{nondegenerediv}}
Let us proceed with the proof of Theorem \ref{nondegenerediv}. It is similar to the one of Theorem \ref{Cor1} (see also \cite[Theorem 3.1]{fm1}), so we sketch it.

As a first step assume that $\mu \equiv 0$ and define, for
$s > 0$, the function
\[
w(t,a,x) := e^{s \Phi (t,a,x)}v(t,a,x)
\]
where $v$ is the solution of \eqref{adjoint} in $\mathcal{V}$; thus,
since $\Phi <0$, $w\in\mathcal{V}$. Of course, $w$ satisfies 
\begin{equation}\label{ND1'}
\begin{cases}
(e^{-s\Phi}w)_t + (e^{-s\Phi}w)_a +(k(x)(e^{-s\Phi}w)_{x})_x  =f(t,x,a), & (t,x) \in
Q,
\\[5pt]
w(0, a, x)= w(T,a, x)= 0, & (a,x) \in Q_{A,1},
\\[5pt]
w(t,A,x)=w(t,0,x)=0, & (t,x) \in  Q_{T,1},
\\[5pt]
w(t, a,0)= w(t, a, 1)= 0, & (t,a) \in  Q_{T,A}.
\end{cases}
\end{equation}
Defining
$
Pw:= w_t + w_a + (kw_x)_x$ and
$P_sw= e^{s\Phi}P(e^{-s\Phi}w),
$
the equation of \eqref{ND1'} becomes
\[
P_sw= P^+_sw + P^-_sw= e^{s\Phi}f,\\
\]
where
\[
P^+_sw := (kw_x)_x
 - s (\Phi_t + \Phi_a) w + s^2k (\Phi_x)^2 w,
\]
and
\[
P^-_sw := w_t+ w_a -2sk\Phi_x w_x -
 s(k\Phi_x)_xw.
\]
Moreover, setting $\langle u, v \rangle_2:= \int_Q uv dxdadt$ and $\| u\|_2:= \int_Q u^2 dxdadt$, one has
\begin{equation}\label{stimettand}
\begin{aligned}
2\langle P^+_sw, P^-_sw\rangle_2 &\le 2\langle P^+_sw, P^-_sw\rangle_2+
\|P^+_sw \|_2^2 + \|P^-_sw\|_2^2\\
& =\| P_sw\|_2^2= \|fe^{s\Phi}\|_2^2.
\end{aligned}
\end{equation}
Proceeding as in the proof of Lemma \ref{lemma1}, one can compute the scalar product $\langle P^+_sw, P^-_sw\rangle_2$, which takes, in this case, the following form
\begin{Lemma}\label{lemma1nd}
The following identity holds:
\begin{equation}\label{D&BTnd}
\left.
\begin{aligned}
&\langle P^+_sw,P^-_sw\rangle_2 \\
&= \frac{s}{2} \int_Q (\Phi_{tt} + \Phi_{aa})w^2dxdadt+ s^3
\int_Q\big(2k \Phi_{xx} + k'\Phi_x\big)k(\Phi_x)^2w^2dxdadt\\
&+s
\int_Q(2k \Phi_{xx} + k'\Phi_x)k(w_x)^2 dxdadt+ s\int_Q k(k\Phi_x)_{xx} w w_xdxdadt\\
& -2s^2 \int_Qk \Phi_x \Phi_{tx}w^2dxdadt   -2s^2 \int_Qk \Phi_x \Phi_{ax}w^2dxdadt   +s \int_Q \Phi_{ta}w^2dxdadt 
\end{aligned}\right\}\;\text{\{D.T.\}$_2$}
\end{equation}
\begin{equation}\nonumber
 \text{\{B.T.\}$_2$}\;\left\{
\begin{aligned}
& -\frac{1}{2}\int_0^A\int_0^{1}\Big[kw_x^2\Big]_0^Tdxda
+\int_0^T\int_0^A\Big[kw_x(w_t+ w_a)\Big]_0^{1}dadt
\\[3pt]&
- s \int_0^T\int_0^A\Big[ k ^2\Phi_xw_x^2\Big]_0^{1}dadt\\[3pt]&
- s \int_0^T\int_0^A\!\!\Big[ k(k\Phi_{x})_xww_x\Big]_0^{1}dadt
+\frac{1}{2}\int_0^A \int_{0}^{1}\!\!\Big[ (s^2k \Phi_x^2 - s \Phi_t-s\Phi_a )w^2\Big]_0^Tdxda
\\[3pt]
&
- s \int_0^T\int_0^A\Big[  (s^2 k^2\Phi_x^3- sk\Phi_x\Phi_t- sk\Phi_x\Phi_a)w^2\Big]_0^{1}dadt\\
&
 -\frac{1}{2}\int_0^T\int_0^{1} \big[kw_x^2\big]_0^Adxdt
 +\frac{1}{2}\int_0^T\int_0^{1}\big[ \big(s^2 k\Phi_x^2 - s (\Phi_t+\Phi_a) \big)w^2\big]_0^Adxdt.
\end{aligned}\right.
\end{equation}
\end{Lemma}
\begin{proof}
Integrating by parts, one has
\[
\langle P^+_sw,P^-_sw\rangle_2= I_1+I_2+ I_3+ I_4,
\]
where
\[
I_1 = \int_Q (kw_x)_x ( w_t -2sk\Phi_x w_x -
 s(k\Phi_x)_xw)dxdadt,
\]
\[
I_2 = \int_Q( - s \Phi_t w + s^2k (\Phi_x)^2 w) ( w_t -2sk\Phi_x w_x -
 s(k\Phi_x)_xw)dxdadt,
\]
\[
I_3= \int_Q  - s \Phi_a w ( w_t -2sk\Phi_x w_x -
 s(k\Phi_x)_xw)dxdadt
\]
and
\[
I_4= \int_Q w_a((kw_x)_x
 - s (\Phi_t + \Phi_a) w + s^2k (\Phi_x)^2 w)dxdadt.
\]
By, \cite[Lemma 3.1]{fm1},
\begin{equation}\label{D&BTndfm}
\left.
\begin{aligned}
&I_1+I_2 \\
&= \frac{s}{2} \int_Q \Phi_{tt} w^2dxdadt+ s^3
\int_Q\big(2k \Phi_{xx} +k'\Phi_x\big)k(\Phi_x)^2w^2dxdadt\\
&- 2s^2 \int_Qk \Phi_x \Phi_{tx}w^2dxdadt  +s
\int_Q(2k \Phi_{xx} + k'\Phi_x)k(w_x)^2 dxdadt\\
&+ s\int_Q k(k\Phi_x)_{xx} w w_xdxdadt
\end{aligned}\right\}\;\text{\{D.T.\}$_3$}
\end{equation}
\begin{equation}\label{D&BTndfm1}
 \text{\{B.T.\}$_3$}\;\left\{
\begin{aligned}
& + \int_0^T\int_0^A[kw_xw_t]_{x=0}^{x=1}da dt- \frac{s}{2}
\int_0^A\int_0^1[w^2\Phi_t]_{t=0}^{t=T}dxda+ \frac{s^2}{2}\int_0^A\int_0^1
[k(\Phi_x)^2 w^2]_{t=0}^{t=T}dxda\\
&-\frac{1}{2}\int_0^A \int_0^1 [k(w_x)^2]_{t=0}^{t=T}dxda
+\int_0^T\int_0^A[-sk(k\Phi_x)_xw w_x]_{x=0}^{x=1}dadt\\
&+ \int_0^T\int_0^A[-s\Phi_x k^2(w_x)^2 +s^2k\Phi_t \Phi_x w^2 - s^3
k^2(\Phi_x)^3w^2 ]_{x=0}^{x=1}dadt.
\end{aligned}\right.
\end{equation}
Next, we compute $I_3$ and $I_4$:
\begin{equation}\label{I3new}
\begin{aligned}
I_3&=\frac{s}{2}\int_Q\Phi_{at}w^2dxdadt -\frac{s}{2} \int_0^A\int_0^1 [\Phi_aw^2]_{t=0}^{t=T}dxda\\
&  -s^2 \int_Q (k\Phi_x)_x \Phi_a w^2dxdadt -s^2 \int_Q k \Phi_x\Phi_{ax}w^2  dxdadt+ s^2\int_0^T\int_0^A[k\Phi_a\Phi_xw^2 ]_{x=0}^{x=1}dadt\\
&+ s^2\int_Q(k\Phi_x)_x \Phi_aw^2  dxdadt\\
&=\frac{s}{2}\int_Q\Phi_{at}w^2dxdadt -\frac{s}{2} \int_0^A\int_0^1 [\Phi_aw^2]_{t=0}^{t=T}dxda\\
& -s^2 \int_Q k \Phi_x\Phi_{ax}w^2  dxdadt+ s^2\int_0^T\int_0^A[k\Phi_a\Phi_xw^2 ]_{x=0}^{x=1}dadt.
\end{aligned}
\end{equation}
On the other hand
\begin{equation}\label{I4new}
\begin{aligned}
I_4&= \int_0^T\int_0^A[kw_xw_a]_{x=0}^{x=1}dadt  -\frac{1}{2}\int_0^T\int_0^1[k(w_x)^2]_{a=0}^{a=A}dxdt\\
&+\frac{1}{2}\int_0^T\int_0^{1}\big[ \big(s^2 k\Phi_x^2 - s (\Phi_t+\Phi_a) \big)w^2\big]_{a=0}^{a=A}dxdt
\\[3pt]
&
+
\frac{s}{2} \int_{Q} \Phi_{aa}w^2dxdadt+
\frac{s}{2} \int_{Q}\Phi_{ta}w^2dxdadt
- s^2\int_{Q} k\Phi_x\Phi_{xa}w^2dxdadt.
\end{aligned}
\end{equation}
Adding \eqref{D&BTndfm}, \eqref{I3new} and \eqref{I4new}, we have the thesis.
\end{proof}

The crucial step is to prove
now the following estimate.

\begin{Lemma}\label{lemma2?}
There
exist two  strictly positive constants $s_0$  and $C$ such that for all $s \ge
s_{0}$ the distributed terms of \eqref{D&BTnd} satisfy the estimate
\[
Cs\int_Q \Theta e^{\kappa \sigma}(w_x)^2 dxdadt + Cs^3
\int_Q\Theta^3 e^{3\kappa \sigma}w^2 dxdadt\le \text{\{D.T.\}$_2$}.
\]
\end{Lemma}
\begin{proof}
By \cite[Lemma 3.2]{fm1}, there
exist two  strictly  positive constants $s_0$  and $C$ such that for all $s \ge
s_{0}$, the distributed terms given in \eqref{D&BTndfm}
\begin{equation}\label{02'}
\begin{aligned}
\frac{C}{2}s\int_Q \Theta e^{\kappa \sigma} (w_x)^2 dxdadt + Cs^3
\int_Q\Theta^3 e^{3\kappa \sigma}w^2 dxdadt\le \text{\{D.T.\}$_3$}.
\end{aligned}
\end{equation}
Moreover, using the definition of $\Phi$, 
the other distributed terms of
\:$<P^+_sw, P^-_sw>_2$ take the form
\begin{equation}\label{02''}
-2 s^2 \int_{Q}k\Theta{\Theta_a}\Psi_x^2 w^2dxdadt
+\frac{s}{2} \int_{Q}\Theta_{aa}\Psi w^2dxdadt
+ s\int_Q\Theta_{ta}\Psi w^2dxdadt.
\end{equation}
Now, the first term in \eqref{02''} can be estimated in the following way:
\begin{equation}\label{ziand}
\begin{aligned}
&\left|2 s^2\mathfrak{d}^2\kappa^2 \int_Q\Theta\Theta_a \frac{e^{2\kappa \sigma}}{k}w^2 dxdadt \right|\le
\frac{2 \mathfrak{d}^2\kappa^2 c}{\min_{[0,1]} k \min_{[0,1]} e^{\kappa \sigma}} s^2\int_Q\Theta^3
e^{3\kappa \sigma}w^2 dxdadt
\\
&\le \frac{C}{6}s^3\int_Q\Theta^3 e^{3\kappa \sigma}w^2 dxdadt, 
\end{aligned}
\end{equation}
for some $C>0$ and $s\geq \displaystyle \frac{12 \mathfrak{d}^2\kappa^2 c}{C\min_{[0,1]} k \min_{[0,1]} e^{2\kappa \sigma}}$. 
Using again \eqref{magtheta}, we have
\begin{equation}\label{ziond}
\begin{aligned}
\left|\frac{s}{2} \int_Q \Theta_{aa}\Psi w^2 dxdadt\right| &\leq s c \max_{[0,1]} |\Psi|\int_Q\Theta^3w^2 dxdadt\leq\frac{cs\max_{[0,1]} |\Psi|}{\min e^{3\kappa \sigma}}\int_Q\Theta^3e^{3\kappa \sigma}w^2dxdadt\\
& \leq \frac{C}{6}s^3\int_Q\Theta^3e^{3\kappa \sigma}w^2dxdadt
\end{aligned}
\end{equation}
and
\[
\left|s\int_Q {\Theta_{ta}}\Psi w^2dxdadt \right| \leq \frac{C}{6}s^3\int_Q\Theta^3e^{3\kappa \sigma}w^2dxdadt,
\]
for $s\geq \displaystyle \sqrt{\frac{6c\max_{[0.1]} |\Psi|}{C\min
e^{3\kappa \sigma}}}$ . 
In conclusion, by the previous inequalities, we obtain
\[
\begin{aligned}
\{D.T.\}_2& \ge  \frac{C}{2}s\int_Q \Theta e^{\kappa \sigma} (w_x)^2 dxdadt + Cs^3
\int_Q\Theta^3 e^{3\kappa \sigma}w^2 dxdadt - \frac{C}{2}s^3
\int_Q\Theta^3 e^{3\kappa \sigma}w^2 dxdadt.
\end{aligned}
\]
Hence, the thesis follows.
\end{proof}
The next lemma holds.
\begin{Lemma}\label{NDBT}
The boundary terms in \eqref{D&BTnd} become
\begin{equation}\label{NDBT1}
\begin{aligned}
\{B.T.\}_2=s\kappa\|k'\|_{L^\infty(0,1)}\int_0^T\int_0^A[k\Theta e^{\kappa\sigma}
(w_x)^2]^{x=1}_{x=0}dadt.
\end{aligned}
\end{equation}
\end{Lemma}

\begin{proof} 
By \cite[Lemma 3.4]{fm1}, the boundary terms given in \eqref{D&BTndfm1} take the form
\[
 \text{\{B.T.\}$_3$}= s\kappa\|k'\|_{L^\infty(0,1)}\int_0^T\int_0^A[k\Theta e^{\kappa\sigma}
(w_x)^2]^{x=1}_{x=0}dadt.
\]
Using the definition of $\Phi$, the other boundary terms of  \:$<P^+_sw,
P^-_sw>_2$ become
\begin{equation}\nonumber
\begin{aligned}
& 
\int_0^T\int_0^A\Big[kw_x w_a\Big]_0^{1}dadt
-\frac{s}{2}\int_0^A \int_{0}^{1}\!\!\Big[ \Phi_a w^2\Big]_0^Tdxda
\\[3pt]
&
+s^2 \int_0^T\int_0^A\Big[ k\Phi_x\Phi_aw^2\Big]_0^{1}dadt
 -\frac{1}{2}\int_0^T\int_0^{1} \big[kw_x^2\big]_0^Adxdt\\
 &
 +\frac{1}{2}\int_0^T\int_0^{1}\big[ \big(s^2 k\Phi_x^2 - s (\Phi_t+\Phi_a) \big)w^2\big]_0^Adxdt\\
 &= \int_0^T\int_0^A\Big[kw_x w_a\Big]_0^{1}dadt-\frac{s}{2}\int_0^A \int_{0}^{1}\!\!\Big[ \Theta_a \Psi w^2\Big]_0^Tdxda\\[3pt]
&
+s^2 \int_0^T\int_0^A\Big[ k\Theta \Theta_a\Psi\Psi_xw^2\Big]_0^{1}dadt
 -\frac{1}{2}\int_0^T\int_0^{1} \big[kw_x^2\big]_0^Adxdt\\
 &
 +\frac{1}{2}\int_0^T\int_0^{1}\big[ \big(s^2 k\Theta^2\Psi_x^2 - s (\Theta_t+\Theta_a)\Psi \big)w^2\big]_0^Adxdt.
\end{aligned}
\end{equation}
As before, since $w \in \mathcal{V}$, $w(0, a, x)$, $w(T,a, x)$, $w_x(t,0, x)$,  $w_x(t,A, x)$, $w(t,A, x)$, $w(t, 0, x)$, $w(t,a, 0)$, $w(t,a,1)$ $w_x(t,a, 0)$ and $w_x(t,a,1)$  are well defined. Moreover, we have that $w_a
(t,a,0)$ and $w_a(t,a, 1)$ make sense and are actually 0.
 Thus, using the boundary conditions of $w=e^{s\Phi}v$, we get
\[
\begin{aligned}
& \int_0^T\int_0^A\Big[kw_x w_a\Big]_0^{1}dadt-\frac{s}{2}\int_0^A \int_{0}^{1}\!\!\Big[ \Theta_a \Psi w^2\Big]_0^Tdxda\\[3pt]
&
+s^2 \int_0^T\int_0^A\Big[ k\Theta \Theta_a\Psi\Psi_xw^2\Big]_0^{1}dadt
 -\frac{1}{2}\int_0^T\int_0^{1} \big[kw_x^2\big]_0^Adxdt\\
 &
 +\frac{1}{2}\int_0^T\int_0^{1}\big[ \big(s^2 k\Theta^2\Psi_x^2 - s (\Theta_t+\Theta_a)\Psi \big)w^2\big]_0^Adxdt=0.
\end{aligned}
\]
Hence the thesis.
\end{proof}
By Lemmas \ref{lemma1nd}, \ref{lemma2?} and \ref{NDBT} the next estimate holds:
\begin{Proposition}\label{stimaND}
There exist two  strictly positive constants $C$ and $s_0$ such
that, for all $s\ge s_0$,  all solutions $w$ of \eqref{ND1'} in $\mathcal{V}$ satisfy
\[
\begin{aligned}
&s\int_Q \Theta e^{\kappa \sigma} (w_x)^2 dxdadt + s^3
\int_Q\Theta^3 e^{3\kappa \sigma}w^2 dxdadt
\\
&\le
C\left(\int_Q\!\!\!f^{2}e^{2s\Phi}dxdadt
- s\kappa\|k'\|_{L^\infty(0,1)}\int_0^T\int_0^A[k\Theta e^{\kappa\sigma}
(w_x)^2]^{x=1}_{x=0}dadt\right).
\end{aligned}
\]
\end{Proposition}
 Recalling the definition of $w$, we have $v= e^{-s\Phi}w$
and $v_{x}=  (w_{x}-s\Phi_{x}w)e^{-s\Phi}$. Thus, 
\[
\begin{aligned}
\left(s\Theta e^{\kappa \sigma}(v_x)^2 + s^3 \Theta^3
e^{3\kappa \sigma} v^2\right)e^{2s\Phi}&\leq c\left[ s\Theta e^{\kappa \sigma}(s^2\Theta^2w^2+(w_x)^2)+s^3 \Theta^3 e^{3\kappa \sigma}w^2\right] \\
& \leq c\left[s\Theta e^{\kappa \sigma}(w_x)^2 + s^3 \Theta^3
e^{3\kappa \sigma} w^2\right],
\end{aligned}
\]
for a  strictly positive constant $c$.
Hence, 
Theorem
\ref{nondegenerediv} follows immediately by Proposition \ref{stima} when $\mu \equiv 0$.

\vspace{0.2cm}
Now, we assume that $\mu \not \equiv 0$.

To complete the proof of Theorem \ref{nondegenerediv} we consider, as before, the function $\overline{f}=f+\mu v$.
 Hence,  there are  two  strictly positive constants $C$ and $s_0$ such that, for
all $s\geq s_0$, the following inequality holds
\begin{equation} \label{fati1ND}
\begin{aligned}
&\int_{Q}(s^{3}\phi^{3}v^{2}+s\phi v_{x}^{2})e^{2s\Phi} dxdadt \leq c \Big(\int_{Q}\bar f^{2}e^{2s\Phi}dxdadt
 -
s\kappa\int_0^T\int_0^A\left[ke^{2s\Phi}\phi(v_x)^2
\right]_{x=0}^{x=1}dadt\Big).
\end{aligned}
\end{equation}
On the other hand, we have
\begin{equation} \label{4'nd}
\begin{aligned}
\int_Q\bar{f} ^{2}e^{2s\Phi}dxdadt
\leq 2\Big(\int_Q|f|^{2}e^{2s\Phi}dxdadt
+\|\mu\|_{L^\infty(Q)}^{2}\int_{Q}|v|^{2}e^{2s\Phi}dxdadt\Big).
\end{aligned}
\end{equation}
Now, applying the Hardy-Poincar\'{e} inequality to the function $\nu:=e^{s\Phi}v$,
we obtain
 \[
      \begin{aligned}
 \int_{Q}|v|^{2}e^{2s\Phi}dxdadt&\le   \int_{Q}
                    \frac{\nu^2}{x^2}dxdadt
    \le
   C\int_{Q}(e^{s\Phi}v)^2_xdxdadt
    \\
    &\le
    C\int_{Q} e^{2s\Phi}v_x^2dxdadt
   + Cs^2\int_{Q}\Theta^2 e^{2s\Phi}\Psi_x^2 v^2dxdadt.
  \end{aligned}
  \]
   Using this last inequality in (\ref{4'nd}), it follows
      \begin{equation}\label{fati2nd}
      \begin{aligned}
      \int_Q |\bar{f}|^{2}e^{2s\Phi}dxdadt
      &\le
      2\int_Q |f|^{2}e^{2s\Phi}dxdadt
      +C\int_Qe^{2s\Phi} v_x^2 dxdadt
      \\&+ Cs^2\int_Q \Theta^2 e^{2s\Phi}e^{2\kappa \sigma} v^2dxdadt.
     \end{aligned} \end{equation}
   Substituting in
   (\ref{fati1ND}), one can conclude

      \[
      \begin{aligned}
    &\int_{Q}(s^{3}\phi^{3}v^{2}+s\phi v_{x}^{2})e^{2s\Phi} dxdadt     \le
       C\Big(\int_Q |f|^{2}e^{2s\Phi}dxdadt
        \\[3pt]& +\int_Q\!\!\!e^{2s\Phi} v_x^2 dxdadt
         + s^2\int_Q\!\!\! \Theta^2 e^{2s\Phi}e^{2\kappa \sigma} v^2dxdadt
          -
s\kappa\int_0^T\int_0^A\left[ke^{2s\Phi}\phi(v_x)^2
\right]_{x=0}^{x=1}dadt \Big).
        \end{aligned}
      \]
     
This completes the proof of Theorem \ref{nondegenerediv}.

\paragraph{Proof of Theorem \ref{nondegenereDebole}}

As before, to prove Theorem \ref{nondegenereDebole}, one 
can assume, first of all, that $\mu\equiv0$. The case $\mu \not \equiv 0$ follows as in the previous subsection. 

If $\mu \equiv0$, the proof in the divergence case is {\it formally} similar to the one  of Theorem \ref{nondegenerediv}  (see also the proof of \cite[Theorem 3.1]{fm1}). Observe that, in this case, integrations by parts are not immediately justified since the function $k$ is not $C^1[0,1]$.   However, proceeding as in \cite{fm1}, one can motivate them. 

In the non divergence case one can proceed as in the proofs of \cite[Theorems 3.1 and 3.2]{fm1}. 
\section*{Acknowledgments}

The author is a member of the Gruppo Nazionale per l'Analisi Matematica, la Probabilit\`a e le loro Applicazioni (GNAMPA) of the
Istituto Nazionale di Alta Matematica (INdAM) and she is partially supported by the reaserch project GNAMPA 2017 {\em Comportamento asintotico e controllo di equazioni di evoluzione non lineari}. 

She dedicates this work to her father.

\end{document}